\documentclass[11 pt,a4paper,twoside]{article}
\usepackage{amsmath,amssymb,amsthm,epsfig,graphics,amsfonts}
\usepackage[latin1]{inputenc}
\usepackage[francais,english]{babel}
\usepackage[all]{xypic}

\def\supp{\mathrm{Supp}}

\newcommand{\NN}{\mathbb{N}}

\newcommand{\QQ}{\mathbb{Q}}
\newcommand{\RR}{\mathbb{R}}
\renewcommand{\SS}{\mathbb{S}}
\newcommand{\TT}{\mathbb{T}}

\newcommand{\ZZ}{\mathbb{Z}}

\def\cA{{\cal A}}   \def\cM{{\cal M}} 
\def\cB{{\cal B}}    
\def\cC{{\cal C}}    
\def\cD{{\cal D}}   \def\cP{{\cal P}} \def\cV{{\cal V}}
  \def\cK{{\cal K}} \def\cQ{{\cal Q}} 
    \def\cX{{\cal X}}

\def\ra{\rightarrow}

\renewcommand{\phi}{\varphi}
\renewcommand{\epsilon}{\varepsilon}

\newtheorem{theoalpha}{Proposition}
\newtheorem{propalpha}[theoalpha]{Proposition}

\newtheorem{theodeuxbis}{Theorem}

\newtheorem{theounbis}{Theorem}

\newtheorem{theounter}{Theorem}

\newtheorem{maintheo}{Theorem}
\newtheorem*{coro*}{Corollary}
\newtheorem*{theo*}{Theorem}

\newtheorem{theo}{Theorem}

\newtheorem*{lemma*}{Lemma}
\newtheorem{lemma}{Lemma}

\newtheorem{sublemma}{Sub-lemma}

\newtheorem*{fact}{Fact}

\theoremstyle{definition}
\newtheorem*{defi}{Definition}
\newtheorem*{rema*}{Remark}

\begin{document}
\sloppy

\title{Realisation of measured dynamics as uniquely ergodic minimal homeomorphisms on manifolds}
\author{F. B\'eguin\footnote{Laboratoire de math\'ematiques (UMR 8628),  Universit\'e Paris Sud,
91405 Orsay Cedex, France.},
\setcounter{footnote}{3}
S. Crovisier\footnote{CNRS et LAGA (UMR 7539), Universit\'e Paris 13, Avenue J.-B. Cl\'ement, 93430 Villetaneuse, France.}~~and F. Le Roux\footnotemark[1]}

\date{}
\maketitle

\begin{abstract}
We prove that the family of  measured dynamical systems which can be  realised as uniquely ergodic minimal homeomorphisms on a given manifold (of dimension at least two) is stable under measured extension. As a corollary, any ergodic system with an irrational eigenvalue is isomorphic to a uniquely ergodic minimal homeomorphism on the two-torus. The proof uses the following improvement of Weiss relative version of Jewett-Krieger theorem: any extension between two ergodic systems is isomorphic to a skew-product on  Cantor sets. 
\end{abstract}

\small
\paragraph{AMS classification} 
37A05 (Measure-preserving transformations),
54H20 (Topological dynamics),
37E30 (Homeomorphisms and diffeomorphisms of planes and surfaces).
\normalsize

\setcounter{section}{0}


\section*{Introduction}
\addcontentsline{toc}{section}{Introduction}

A natural problem in dynamical systems is to determine which measured dynamics admit topological or smooth realisations. Results in this direction include:
\begin{itemize}
\item[--] constructions of smooth diffeomorphisms on manifolds satisfying some specific ergodic properties (see for example~\cite{FayKat}),
\item[--] general results about topological realisations on Cantor sets (Jewett-Krieger theorem and its generalisations; see  below).
\end{itemize}
In this paper, we tackle the following question: \emph{given a manifold $\cM$, which measured dynamical systems can be realised as uniquely ergodic minimal homeomorphisms on $\cM$?}  Our main result asserts that this class of  dynamical systems is stable under extension.
\begin{maintheo}\label{theo.extension}
Let $\cM$ be a compact  topological manifold of dimension at least two. Assume we are given
\begin{itemize}
\item[--]  a uniquely ergodic minimal homeomorphism $F$ on $\cM$, with invariant measure~$m$; 
\item[--] an invertible ergodic dynamical system $(Y,\nu,S)$ on a standard Borel space, which is an extension of~$(\cM,m,F)$.
\end{itemize}
Then there exists  a uniquely ergodic minimal homeomorphism $G$ on $\cM$, with invariant measure $v$, such that the measured dynamical system $(\cM,v,G)$ is isomorphic to $(Y,\nu,S)$.
\end{maintheo}

Let us recall the classical definitions. A \emph{measured  dynamical system} is given by $(X,\mu,R)$ where $(X,\mu)$ is a probability space and $R:X \to X$ is a bi-measurable bijective map for which $\mu$ is an  invariant measure. Given two such systems $(X,\mu,R)$ and $(Y,\nu,S)$, the second is an \emph{extension} of the first through a measurable map $\Phi : Y_{0} \to X_{0}$ if  $X_{0},Y_{0}$ are full measure subsets of $X,Y$, the map $\Phi$ sends the measure $\nu$ to the measure $\mu$, and the conjugacy relation  $\Phi S = R \Phi$ holds on $Y_{0}$. If,  in addition,  the map $\Phi$ is bijective and bi-measurable, then the systems are \emph{isomorphic}.  A measurable space $(X,\cA)$ is called a \emph{standard Borel space} if $\cA$ is the Borel  $\sigma$-algebra of some topology on $X$ for which $X$ is a Polish space (\emph{i.e.} a metrizable complete separable space). Throughout the text all the topological spaces will be implicitly equipped with their Borel $\sigma$-algebra, in particular all the measures on topological spaces are Borel measures.
Note that if some measured dynamical system $(Y,\nu,S)$ satisfies the conclusion of theorem~\ref{theo.extension}, then it is obviously isomorphic to a dynamical system on a standard Borel space, thus it is reasonable to restrict ourselves to such systems.
We say that a homeomorphism $F$ on a topological space $\cM$, with an invariant measure $m$, is a \emph{realisation} of a measured dynamical system $(X,\mu,R)$ if the measured dynamical system $(\cM,m,F)$ is isomorphic to $(X,\mu,R)$. 

Independently of theorem~\ref{theo.extension} we will prove the following result.

\begin{maintheo}
\label{theo.circle-rotation}
Any irrational rotation of the circle admits a uniquely ergodic minimal realisation on the two-torus.
\end{maintheo}
The proof of this result relies on a rather classical construction. The realisation is a skew-product $F(x,y)=(x+\alpha,A(x).y)$ where $A:\SS^1\to\mbox{SL}(2,\RR)$ is a continuous map. It will be obtained as the limit of homeomorphisms conjugated to the ``trivial" homeomorphism \hbox{$(x,y)\mapsto (x+\alpha,y)$.}

Theorems~\ref{theo.extension} and~\ref{theo.circle-rotation} can be associated to provide a partial answer to the realisation problem on the two-torus. Remember that a measured dynamical system $(X,\mu,R)$  is  an extension of some irrational circle rotation if and only if the spectrum of the associated operator on $L^2(X,\mu)$ has an irrational eigenvalue. Therefore, as an immediate consequence of theorem~\ref{theo.extension} and~\ref{theo.circle-rotation}, we get the following corollary.
\begin{coro*}
If an ergodic measured dynamical system on a standard Borel space has an irrational eigenvalue in its spectrum, then it admits a uniquely ergodic minimal realisation on the two-torus.
\end{coro*}

This result is known not to be optimal: indeed there exist uniquely ergodic minimal homeomorphisms of the two-torus that are weakly mixing (see~\cite{Skl,AnoKat}). Actually, it might turn out that \emph{every} aperiodic ergodic system admits a uniquely ergodic minimal realisation on the two-torus; at least no obstruction is known to the authors. For instance, we are not able to answer the following test questions: 
\begin{itemize}
\item[--] does an adding machine admit a uniquely ergodic minimal realisation on the two-torus? 
\item[--] what about a Bernoulli shift?
\end{itemize}

Our original motivation for studying realisation problems was to generalise Denjoy counter-examples in higher dimensions, that is, to construct (interesting) examples of homeomorphisms of the $n$-torus that are topologically semi-conjugate to an irrational rotation. The proof of theorem~\ref{theo.extension} actually provides a topological semi-conjugacy between the maps $G$ and $F$ (see theorem~\ref{theo.extension-bis} below). Thus another corollary of our results is: \emph{any ergodic system which is an extension of an irrational rotation $R$ of the two-torus can be realised as a uniquely ergodic minimal homeomorphism which is topologically semi-conjugate to $R$.} 
For more comments on realisation problems and generalisations of Denjoy counter-examples, we refer to the introduction of our previous work~\cite{BegCroLeR}.
The reader interested in the smooth version can consult~\cite{FayKat}, and especially the last section.

The proof of theorem~\ref{theo.extension} relies heavily on our previous work. Given the main result in~\cite{BegCroLeR}, theorem~\ref{theo.extension} essentially reduces to a realisation problem on Cantor sets: roughly speaking, we have to turn a measured extension into a topological skew-product. Concerning this issue, remember that the classical realisation problem on Cantor sets has  been solved by Hansel-Raoult and Krieger.
 Improving the work of Jewett, they have proved the so-called Jewett-Krieger theorem:   \emph{any ergodic system admits a  uniquely ergodic minimal realisation on a Cantor set} (\cite{Jewett,Hansel-Raoult,Krieger}). A relative version of this result has later been provided by Weiss (\cite{Weiss,Glasner,GlaWei}), allowing the realisation of a measured extension as a topological semi-conjugacy. In the present text we improve Weiss result by showing that the topological semi-conjugacy can be turned into a topological skew product between Cantor sets, as expressed by the following theorem.
 
\begin{maintheo}[Fibred Jewett-Krieger theorem]
\label{theo.fibred-weiss}
Suppose that we are given
\begin{itemize}
\item[--] a uniquely ergodic minimal homeomorphism $f$ on a Cantor set $\cK$, with invariant measure $\mu$;
\item[--] an ergodic dynamical system $(Y,\nu,S)$ on a standard Borel space,  which is an extension of $(\cK,\mu,f)$ through a map $p$.
\end{itemize}
Then there exist a Cantor set $\cQ$, a uniquely ergodic minimal homeomorphism $g$ on the product $\cK \times \cQ$ with invariant measure $\nu'$, 
a full measure subset $Y_0$ of $Y$ and a map $\Phi: Y_{0} \ra \cK \times \cQ$, 
 such that the map $\Phi$  is an isomorphism between $(Y,\nu,S)$ and $(\cK \times \cQ,\nu',g)$,
 and the following diagram commutes (where $\pi_{1} : \cK \times \cQ \ra \cK$ is the first projection).
$$
\xymatrix{ 
&   \cK \times \cQ   \ar@{.>}[rr]  ^g    \ar@{.>}'[d] [ddd] ^(.3){\pi_1} & & \cK \times \cQ   \ar@{.>}[ddd]  ^(.55){\pi_1} \\ 
Y_{0} \ar@{.>}[ur]    ^{\Phi}   \ar@{>}[rr]   ^(.65){S}    \ar@{>}[ddr]   ^(.4){p}  & & Y_{0} \ar@{.>}[ur]    ^{\Phi}   \ar@{>}[ddr]   ^(.4){p} \\
\\
& \cK \ar@{>}^{f}[rr]    & & \cK\\
}$$
\end{maintheo}

The proof of theorem~\ref{theo.extension}  will  actually involve a non-compact situation which will require  a ``bi-ergodic'' version of this result,  provided as theorem~\ref{theo.bi-ergodic-fibred} below. From a strictly logical viewpoint, we will not use theorem~\ref{theo.fibred-weiss}. We give the statement here since it seems to be interesting for its own sake; furthermore,  the ``bi-ergodic" version will be presented as an adaptation of this simpler statement.

\paragraph{Organisation of the paper}
In the next paragraph, using the results of~\cite{BegCroLeR},  we reduce theorem~\ref{theo.extension} to a  ``bi-ergodic fibred version of Jewett-Krieger theorem" (theorem~\ref{theo.bi-ergodic-fibred} below).
Then, theorem~\ref{theo.bi-ergodic-fibred} (and also theorem~\ref{theo.fibred-weiss}) is proved in parts~\ref{A} and~\ref{B}. Part~\ref{C} deals with the realisation of circle rotations and gives the proof of theorem~\ref{theo.circle-rotation} .


\paragraph{Acknowledgements} We are grateful to Jean-Paul Thouvenot for his constant interest in our  work.

\section*{From topological realisations on Cantor sets to topological realisations on manifolds}
\addcontentsline{toc}{section}{Strategy}

We will now explain how to deduce theorem~\ref{theo.extension} from the results of~\cite{BegCroLeR} and a ``bi-ergodic fibred version of Jewett-Krieger theorem". Actually, we will obtain the following more precise version of theorem~\ref{theo.extension}.
\begin{theounbis}\label{theo.extension-bis}
Let $\cM$ be a compact  topological manifold whose dimension is at least two. 
Assume we are given
\begin{itemize}
\item[--] a uniquely ergodic minimal homeomorphism $F$ on $\cM$ with invariant measure $m$;
\item[--] an ergodic system $(Y,\nu,S)$ on a standard Borel space, which is an extension of $(\cM,m,F)$ through a map $p$.
\end{itemize}
Then there exist a uniquely ergodic minimal homeomorphism $G$ on $\cM$ with unique invariant measure $v$, 
a surjective continuous map $P:\cM\to\cM$,
a full measure subset $Y_0$ of $Y$ and a map $\Phi: Y_{0}  \ra \cM$,
such that  the map $\Phi$ is an isomorphism between $(Y,\nu,S)$ and $(\cM,v,G)$,
and the following diagram commutes.
$$
\xymatrix{ 
&  \cM   \ar@{.>}[rr]  ^G    \ar@{.>}'[d] [ddd] ^(.3){P} & & \cM \ar@{.>}[ddd]  ^(.55){P} \\ 
Y_{0} \ar@{.>}[ur]    ^{\Phi}   \ar@{>}[rr]   ^(.65){S}    \ar@{>}[ddr]   ^(.4){p}  & & Y_{0} \ar@{.>}[ur]    ^{\Phi}   \ar@{>}[ddr]   ^(.4){p} \\
\\
& \cM \ar@{>}^{F}[rr]    & & \cM\\
}$$
\end{theounbis}

Before entering the proof of theorem~\ref{theo.extension-bis}, let us give a sketch. 
Consider a  homeomorphism $F$ on $\cM$ with unique invariant measure $m$ as in the statement.
 In~\cite{BegCroLeR} we constructed a Cantor subset $K$ of $\cM$ with some properties called \emph{dynamical coherence}. These properties allow us to endow the union of all iterates of $K$ with a natural topology that makes it homeomorphic to a Cantor set $\cK$ minus one point denoted by $\infty$. Then $F$ induces a  homeomorphism $f$ on $\cK$ which has exactly two ergodic measures $\mu$ and $\delta$, with $(\cK,  \mu, f)$  isomorphic to $(\cM,m,F)$, and $\delta$ the Dirac measure at the point $\infty$. 
 Now consider the extension $(Y,\nu,S)$ of $(\cM,m,F)$ as in theorem~\ref{theo.extension-bis}. The ``bi-ergodic fibred Jewett-Krieger theorem'' proved in parts~\ref{A} and~\ref{B} tells us that  $(Y,\nu,S)$ is isomorphic to a system $(\cK \times \cQ, \nu', g)$, where $g$ is a topological skew-product   over $f$, and $\nu'$ is an ergodic measure that lifts $\mu$. Then the main  result in~\cite{BegCroLeR} asserts that the system $(\cK \times \cQ, \nu', g)$ is realisable as a uniquely ergodic minimal homeomorphism $G$ on $\cM$. This homeomorphism satisfies the conclusions of theorem~\ref{theo.extension-bis}.

\subsection*{Proof of theorem~\ref{theo.extension-bis} (assuming the bi-ergodic fibred Jewett-Krieger theorem~\ref{theo.bi-ergodic-fibred})}
We assume the hypotheses of the theorem~\ref{theo.extension-bis}.
Let us consider a measurable set $A_0\subset M$ that has positive $m$-measure.
Proposition 2.10 of~\cite{BegCroLeR} provides us with a Cantor set $K\subset A_0$
that has positive $m$-measure and that is \emph{dynamically coherent}:
\begin{itemize}
\item[--] for every integer $n$, the intersection $F^n(K)\cap K$ is open in $K$;
\item[--] for every integer $p\geq 1$ and every point $x\in K$, there exists
$p$ consecutive positive iterates $F^{k}(x),\dots, F^{k+p-1}(x)$
and $p$ consecutive negative iterates $F^{-\ell}(x),\dots, F^{-\ell -p+1}(x)$ outside $K$.
\end{itemize}
We consider the set
$$Z:=\bigcup_{n\in \ZZ} F^n(K).$$
We endow $Z$ with the following topology:
a set $O\subset Z$ is open if and only if $O\cap F^n(K)$ is
open in $F^n(K)$ for every $n\in \ZZ$. In particular,
the injection $j\colon Z\to \cM$ is continuous.
The dynamical coherence (first item) implies that any open set of $F^n(K)$
for the (usual) relative topology induced by the topology of $\cM$
is also an open subset of $Z$.
In particular, $Z$ is locally compact,
the restriction of $F$ to $Z$ is still a homeomorphism
and $j$ is a bi-measurable bijection from $Z$ onto its image.
The dynamical coherence (second item) implies that $Z$ is not compact.
We consider the Alexandroff one-point compactification, that is,
the space $\cK=Z\cup\{\infty\}$ whose open sets are
the open sets of $Z$ and the complements of compact subsets of $Z$.

\begin{fact}
The topological space $\cK$ is a Cantor set.
\end{fact}

\begin{proof}
Let us consider the finite unions $K_n=\bigcup_{|k|\leq n} F^k(K)$ for each $n\geq 0$.
Each set $F^n(K)\subset \cK$ is a Cantor set; the intersection of
any iterates $F^{n_1}(K),\dots,F^{n_s}(K)$ is clopen (\emph{i.e.} closed and open) in the Cantor set $F^{n_1}(K)$,
hence is a Cantor set. From this, one deduces
that each $K_n$ is a Cantor set.
By construction $K_{n+1}\setminus K_n$
is non-empty ($Z$ is not compact, hence strictly contains $K_n$),
clopen in $K_n$ and thus is also a Cantor set. From this it is easy to build a homeomorphism $h$ between $\cK$ and the
standard Cantor set $\Sigma=\{0,1\}^\NN$: $h$ maps $K_{0}$ 
 homeomorphically onto  $\{0\}\times \{0,1\}^\NN$, and it maps
  $K_{n+1} \setminus K_{n}$  to $\{1\}^{n}\times \{0\}\times \{0,1\}^\ZZ$
for $n\geq 0$; it sends $\infty$ on the sequence $(1,\dots,1,\dots)$.
Since the complements of the $K_n$'s define a basis of neighbourhoods of $\infty$
in $\cK$, the map $h$ is continuous at $\infty$.
Hence $h$ is a homeomorphism and $\cK$ is a Cantor set.
\end{proof}

We extend the homeomorphism $F$ on $Z$ to a homeomorphism
$f$ on $\cK$ by setting $f(\infty)=\infty$.
Since $m$ is ergodic the union $\bigcup_{n\in \ZZ}F^n(K)\subset \cM$
has full $m$-measure. The measure $m$ induces 
an ergodic aperiodic $f$-invariant measure $\mu$ on $\cK$.
Since $F$ is minimal the measure $\mu$ has full support on $\cK$ (\emph{i.e.} every open set has positive measure). Since $F$ is uniquely ergodic, the only other $f$-invariant ergodic measure on $\cK$ is the Dirac measure~$\delta_\infty$. We now state our ``bi-ergodic fibred version of Jewett-Krieger theorem'', whose proof will occupy parts~\ref{A} and~\ref{B} of the paper (see the end of part~\ref{B}).
\pagebreak
\begin{maintheo}[``Bi-ergodic fibred Jewett-Krieger theorem'']
\label{theo.bi-ergodic-fibred}
Suppose that we are given
\begin{itemize}
\item[--] a homeomorphism $f$ on a Cantor set $\cK$,  which admits exactly two ergodic invariant  measures: a measure $\mu$ with full support and a Dirac measure $\delta _{\infty}$  at a point of $\cK$;
\item[--] an ergodic measured dynamical system $(Y,\nu,S)$ on a standard Borel space,
and a measurable map $p : Y \ra \cK$ such that $p_*\nu=\mu$ and $f \circ p = p \circ S$.
\end{itemize}
Then there exist a Cantor set $\cQ$, a homeomorphism $g$ on  $\cK\times \cQ$,
 an ergodic measure $\nu'$ for $g$,
a full measure subset $Y_0$ of $Y$ and a map $\Phi:Y_{0} \to \cK\times \cQ$, such that
$\nu'$ is the only $g$-invariant measure satisfying  $(\pi_1)_*\nu'=\mu$ (where $\pi_1:\cK\times \cQ\to\cK$ is the first projection map), 
the map $\Phi$ is  an isomorphism between $(Y,\nu,S)$ and $(\cK\times \cQ,\nu',g)$,
and  the following diagram commutes.
$$
\xymatrix{ 
&   \cK \times \cQ   \ar@{.>}[rr]  ^g    \ar@{.>}'[d] [ddd] ^(.3){\pi_1} & & \cK \times \cQ   \ar@{.>}[ddd]  ^(.55){\pi_1} \\ 
Y_{0} \ar@{.>}[ur]    ^{\Phi}   \ar@{>}[rr]   ^(.65){S}    \ar@{>}[ddr]   ^(.4){p}  & & Y_{0} \ar@{.>}[ur]    ^{\Phi}   \ar@{>}[ddr]   ^(.4){p} \\
\\
& \cK \ar@{>}^{f}[rr]    & & \cK\\
}$$
\end{maintheo}


Applying theorem~\ref{theo.bi-ergodic-fibred}, we get a Cantor set $\cQ$, a homeomorphism $g$ on $\cK\times\cQ$, an ergodic measure $\nu'$, an isomorphism $\Phi$ between $(Y,\nu,S)$ and $(\cK\times\cQ,g,\nu')$. The restriction of the map $g$  to $\bigcup_{n\in \ZZ} F^n(K)\times \cQ$, which we still denote by  $g$,  satisfies the following properties.
\begin{enumerate}
\item The map $g$ is a bi-measurable bijection of $\bigcup_{n\in \ZZ} F^n(K)\times \cQ$
that fibres above the restriction of  $F$ to $\bigcup_{n\in \ZZ} F^n(K)$.
\item For every $x\in \bigcup_{n\in \ZZ} F^n(K)$, the map $g_{x} : (x,c)\mapsto g(x,c)$ is a homeomorphism from $\{x\}\times \cQ$ to $\{F(x)\}\times \cQ$.
\item For every integer $n$, the map $g^n$ is continuous on $\cK\times \cQ$.
\end{enumerate}
We are now in a position to apply the following result of~\cite{BegCroLeR}.

\begin{theo*}[Theorem 1.3 of \cite{BegCroLeR} in the uniquely ergodic case]
Let $F$ be a uniquely ergodic homeomorphism on a compact manifold $\cM$ of dimension $d\geq 2$, with invariant measure $m$.  Assume that $m$ gives positive measure to some measurable subset $K$ of $\cM$. Let $g$ be a bi-measurable bijection of $\bigcup_{n\in \ZZ} F^n(K)\times \cQ$, with a unique invariant measure $\nu'$, and  which satisfies the above properties~1, 2, 3.

Then there exists a uniquely ergodic homeomorphism $G$ on $\cM$ with invariant measure~$v$, such that $(\cM,v,G)$ is  isomorphic to  $\left(\bigcup_{n\in\ZZ} F^n(K)\times \cQ, \nu',g\right)$, and a continuous surjective map $P:\cM\to\cM$ such that $P\circ G=F\circ P$.
Furthermore if $F$ is minimal then $G$ can be chosen to be minimal.
\end{theo*}

By composition of the isomorphisms provided by this theorem and theorem~\ref{theo.bi-ergodic-fibred}, we get the isomorphism between $(Y,\nu,S)$ and $(\cM,v,G)$ required by theorem~\ref{theo.extension-bis}. This almost completes the proof of the theorem: we just have to check the commutation of the triangular faces in the diagram appearing in  the statement of theorem~\ref{theo.extension-bis}. 
In our setting the Cantor set $K$ has the two following additional properties introduced in~\cite{BegCroLeR}: it is dynamically coherent and dynamically meagre (this last property is automatic when $F$ is minimal).
In this case, the proof of  theorem~1.3 in \cite{BegCroLeR} provides an isomorphism $\Phi$ between  $\left(\bigcup_{n\in\ZZ} F^n(K)\times \cQ, \nu',g\right)$
and $(\cM,v,G)$
defined on the whole set $\bigcup_{n\in \ZZ} F^n(K)\times \cQ$,
and which  satisfies $\Phi(x\times \cQ)\subset P^{-1}(x)$ for every $x$ in $K$ (see lemma~5.2 and proposition~8.2 in~\cite{BegCroLeR}).
 Consequently, the following diagram commutes.
$$
\xymatrix{ 
&  \cM \ar@{.>}[rr]  ^G    \ar@{.>}'[d] [ddd] ^(.3){P} & & \cM  \ar@{.>}[ddd]  ^{P} \\ 
\bigcup_{n\in\ZZ} F^n(K)\times \cQ \ar@{.>}[ur]    ^{\Phi}   \ar@{>}[rr]   ^(.55){g}    \ar@{>}[ddd]   ^(.4){\pi_1}  & & \bigcup_{n\in\ZZ} F^n(K)\times \cQ \ar@{.>}[ur]    ^{\Phi}   \ar@{>}[ddd]   ^(.4){\pi_1} \\
\\
& \cM \ar@{>}'[r]^(.8){F}[rr]    & & \cM \\
\bigcup_{n\in\ZZ} F^n(K)\ \ar@{>}[rr]   ^{F=f}  \ar@{>}[ur]  ^{i}  & & \bigcup_{n\in\ZZ} F^n(K)\ \ar@{>}[ur]  ^{i} 
}$$
The sets $\bigcup_{n\in\ZZ} f^n(K)\times \cQ$ and $\bigcup_{n\in\ZZ} f^n(K)$ are full measure subsets respectively in $\cK \times \cQ$ and $\cK$. Thus we can concatenate the above cubic diagram with the prismatic diagram provided by theorem~\ref{theo.bi-ergodic-fibred} to get the diagram required by theorem~\ref{theo.extension-bis}. In particular, the triangular faces commute. This completes the proof of theorem~\ref{theo.extension-bis} (assuming theorem~\ref{theo.bi-ergodic-fibred}).

\subsection*{A more general version of theorem~\ref{theo.extension-bis}}

The same construction as above, using the full statement of theorem 1.3 of~\cite{BegCroLeR}, also works with non-uniquely ergodic maps. We just give the resulting statement, leaving the details to the reader. Let us first recall some definitions from~\cite{BegCroLeR}. An $S$-invariant measurable subset $X_{0}$ of a measurable dynamical system $(X,S)$ is \emph{universally full} if it has full measure for any $S$-invariant probability measure on $X$. Two measurable systems $(X,R)$  and $(Y,S)$  are  \emph{universally isomorphic} through a bijective bi-measurable map $\Phi : X_{0} \ra Y_{0}$ if $X_{0}$ (resp. $Y_{0}$) is an  $R$ (resp. $S$)-invariant universally full subset in $X$ (resp. $Y$), and the equality $\Phi \circ R = S \circ \Phi$ holds on $X_{0}$.

\begin{theounter}
\label{t.principal-technique}
Let $F$ be a homeomorphism on a compact topological manifold $\cM$
whose dimension is at least two and $m$ be an aperiodic ergodic measure
for $F$. Let $(Y,\nu,S)$ be an ergodic  system on a standard Borel space which is an extension of $(\cM,m,F)$.

Then there exists a homeomorphism $G$ of $\cM$
such that $(\cM,G)$ is universally  isomorphic to the disjoint union
$$ \left(\cM\setminus Z,F\right) \sqcup\left(\bar Y,\bar S\right),
$$
where $Z\subset \cM$ is an $F$-invariant measurable set that has full
$m$-measure and zero measure for any other ergodic $F$-invariant measure
and where $(\bar Y,\bar S)$ is a uniquely ergodic realisation of $(Y,\nu,S)$.

Moreover $G$ is a topological extension of $F$: there exists a continuous
surjective map $\Phi\colon \cM\to \cM$ which is one-to-one outside $\Phi^{-1}(\supp(m))$
and satisfies $\Phi\circ G=F\circ \Phi$.
If $F$ is minimal (resp. transitive), then $G$ can be chosen minimal (resp. transitive).
\end{theounter}


\part{Open bi-ergodic relative Jewett-Krieger theorem}\label{A}

Parts~\ref{A} and~\ref{B} are devoted to the proof of theorem~\ref{theo.bi-ergodic-fibred}. Roughly speaking, we have to show that ergodic extensions can be turned into topological skew-products (on Cantor sets). This will be achieved in two steps. In part~\ref{A}, we prove that an ergodic extension can be turned into an ``open topological semi-conjugacy" (\emph{i.e.} a topological semi-conjugacy where the conjugating map is open). Then, in part~\ref{B}, we will show that every open topological semi-conjugacy is ``isomorphic'' to a skew-product. More precisely the aim of part~\ref{A} is to prove the following result.

\begin{propalpha}\label{p.bi-ergodic-ouvert}
Suppose that we are given
\begin{itemize}
\item[--] a homeomorphism $f$ on a Cantor set $\cK$, 
which admits exactly two ergodic measures $\mu$ and $\delta _{\infty}$,
where $\mu$ has full support and $\delta _{\infty}$ is a Dirac measure on a  point $\infty \in \cK$;
\item[--] an invertible ergodic  measured dynamical system $(Y,\nu,S)$ on a standard Borel space, and a measurable map $p : Y \ra \cK$ such that $p_{*}\nu=\mu$ and $f \circ p = p \circ S$.
\end{itemize}
Then there exist a homeomorphism $g$  on a Cantor set $\cC$
which preserves an ergodic measure $\nu'$,
an isomorphism $\Phi$ between $(Y,\nu,S)$ and $(\cC, \nu',g)$
defined on a full measure subset $Y_0$ of $Y$,
and an onto continuous map $\Pi\colon \cC\to \cK$
such that the following diagram commutes.
$$
\xymatrix{ 
&   \cC  \ar@{.>}[rr]  ^g    \ar@{.>}'[d] [ddd] ^(.3){\Pi} & & \cC  \ar@{.>}[ddd]  ^(.55){\Pi} \\ 
Y_{0} \ar@{.>}[ur]    ^{\Phi}   \ar@{>}[rr]   ^(.65){S}    \ar@{>}[ddr]   ^(.4){p}  & & Y_{0} \ar@{.>}[ur]    ^{\Phi}   \ar@{>}[ddr]   ^(.4){p} \\
\\
& \cK \ar@{>}^{f}[rr]    & & \cK\\
}$$
Furthermore, the following properties can be satisfied.
\begin{itemize}
\item[--] The map $\Pi$ is open.
\item[--] $\Pi^{-1}(\infty)$ is a singleton set and $g$ preserves exactly two ergodic measures: the first one is $\nu'$ and the other one is the Dirac measure on $\Pi^{-1}(\infty)$.
\end{itemize}
\end{propalpha}

This result is a variation on Weiss's relative version of Jewett-Krieger theorem. In the following we get four successive versions. The first one is Weiss's original statement which concerns the uniquely ergodic case and does not provides the openness of the semi-conjugacy. The second one add the openness.
The third one is the bi-ergodic version without the openness. The last one
is proposition~\ref{p.bi-ergodic-ouvert}.
The proofs are given in sections~\ref{s.weiss} to~\ref{sec.bi-ergodic-open}
which have similar structures.
The last three proofs are explained as modifications of the first one.
A key notion is that  of uniform algebra, already used
by Hansel and Raoult for obtaining Jewett-Krieger theorem.
In the first section we try to follow the ideas from~\cite{GlaWei}, while providing the much more detailed exposition required by our three future modifications.

\section{Relative Jewett-Krieger theorem (Weiss theorem)}\label{s.weiss}

\begin{theo}[Weiss]\label{theo.weiss}
Suppose that we are given
\begin{itemize}
\item[--] a uniquely ergodic minimal homeomorphism $f$ on a Cantor set $\cK$, with  invariant  measure $\mu$,
\item[--] an ergodic measured dynamical system $(Y,\nu,S)$ on a standard Borel space and a measurable map $p : Y \ra \cK$
such that $p_{*}\nu=\mu$ and $f \circ p = p \circ S$.
\end{itemize}
Then there exist a uniquely ergodic minimal homeomorphism $g$ on a Cantor set $\cC$ with invariant measure $\nu'$, an isomorphism $\Phi$ between $(Y,\nu,S)$ and $(\cC, \nu',g)$,
an onto continuous map $\Pi:\cC \to \cK$, and a full measure subset $Y_{0}$ of $Y$, such that we have the same
commutative diagram as in proposition~\ref{p.bi-ergodic-ouvert}.
\end{theo}

\subsection{Strategy}
\label{ss.strategy}
Proving Weiss's theorem essentially amounts to constructing an $S$-invariant algebra\footnote{Remember that an algebra is closed under  \emph{finite} intersections or unions, as opposed to the \emph{countable} operations permitted in a $\sigma$-algebra. Here the distinction between algebras and $\sigma$-algebras is crucial.}
 $\cA_0$ on $Y$, which generates the $\sigma$-algebra of $Y$ (modulo a null-set), and such that all the elements of $\cA_0$ are \emph{uniform sets}. Roughly speaking, a set $A\subset Y$ is uniform if the Birkhoff averages (for the map $S$) of the characteristic function of $A$ converge uniformly towards $\nu(A)$. Let us make this precise.

\begin{defi}
Let $(Y,\nu,S)$ be an ergodic measured dynamical system. A set\footnote{All sets are implicitly supposed to be measurable with respect to the $\sigma$-algebra on $Y$.} $A\subset Y$ is $(\varepsilon,N)$-\emph{uniform}  if
for $\nu$-a.e. $y\in Y$, for every $n \geq N$, the proportion of the finite orbit $y, \dots, S^{n-1}(y)$ in $A$ is approximately  equal to  $\nu(A)$, with error strictly less than  $\varepsilon$: in other words, the Birkhoff sum $S_{n}(1_{A},y)$
for the characteristic function of $A$ satisfies 
$$\left|\frac 1 n S_{n}(1_{A},y) - \nu(A)\right|< \varepsilon$$
for $\nu$-a.e. $y\in Y$ and every $n\geq N$. A set $A\subset Y$  is \emph{uniform} if for every $\varepsilon >0$ there exists $N> 0$ such that $A$ is $(\varepsilon,N)$-uniform. An algebra or a $\sigma$-algebra on $Y$ is uniform if it consists of uniform sets.
\end{defi}
The property of being uniform is stable under disjoint union and complement; note however that in general it is not  stable under intersection nor union.
Thus the family of all uniform sets does not constitutes an algebra.

Note that if $f$ is a uniquely ergodic minimal homeomorphism of a Cantor set $\cK$ with invariant measure $\mu$, then the clopen (\emph{i.e.} closed and open) subsets of $\cK$ are uniform sets for $(\cK,\mu,f)$. Thus, under the hypotheses of Weiss theorem, the elements of  the algebra
$$
\quad\quad\quad\quad\quad\quad\quad\quad\quad\quad
\cB = \{p^{-1}(K), K \subset \cK \mbox{ is a clopen set\}}\quad\quad\quad\quad\quad\quad\quad\quad\quad\quad (*)
$$
are uniform sets for $(Y,S,\nu)$. Moreover, this algebra $\cB$
\emph{has no atom}: every $\cB$-measurable set of positive measure can be partitioned into  $\cB$-measurable subsets of arbitrarily small positive measure. These are the hypotheses of the following lemma.
\pagebreak
\begin{lemma}\label{lem.uniform-algebra}
Let $(Y,\nu,S)$ be an ergodic measured dynamical system on a standard Borel space. Suppose $\cB$ is an algebra on $Y$ such that \nopagebreak
 \begin{enumerate}
\item $\cB$ is a countable family of measurable sets and has no atom;
\item $\cB$ is $S$-invariant and its elements are uniform sets.
\end{enumerate}
Then there exists  an algebra $\cA_{0}$ which contains $\cB$, which still satisfies properties 1 and 2 above (in particular the elements of $\cA_0$ are uniform sets), and which in addition generates the $\sigma$-algebra of $Y$ modulo a null-set.
\end{lemma}

Remember that the algebra $\cA_{0}$ \emph{generates the $\sigma$-algebra of $Y$ modulo a null-set} if there exists a full measure set $Y_{0}\subset Y$ such that the $\sigma$-algebra generated by the algebra $\cA_{0 \mid Y_{0}}:= \{a \cap Y_{0}, a \in \cA_{0}\}$ coincides with the $\sigma$-algebra of $Y$ induced on $Y_0$.

\begin{rema*}
Let $\cB$ be an algebra satisfying the hypotheses of the lemma. Let $Y'$ be the complementary set of the union of the null-sets (zero measure sets) in $\cB$.
Then  $(Y',\nu_{\mid Y'},S_{\mid Y'})$ is an ergodic system on a standard Borel space (see~\cite{Kechris}). The algebra $\cB' = \cB_{\mid Y'}$ again satisfies the 
hypotheses of the lemma, and has the following additional property: it contains no non-empty null-set. If we find an algebra $\cA_{0}'$ satisfying the conclusion of the lemma with respect to $\cB'$, then the algebra $\cA_{0}$ generated by $\cA_{0}' \cup \cB$ will work for $\cB$. This shows that we can assume, in the proof of the lemma, that $\cB$ has no non-empty null-set.
Under this additional property, each element $B$ of $\cB$ shares a stronger uniformity property: if $B$ is  $(\varepsilon,N)$-uniform, then the proportion of the orbit $y, \dots, S^{n-1}(y)$ in $A$ is equal to   $\nu(A) \pm \varepsilon$  for \emph{every} (and not only almost every) $y\in Y$ and every $n \geq N$. Indeed the set of ``bad'' points for fixed time $n \geq N$ is a $\cB$-measurable null-set, and thus it is empty.
In particular, for every non-empty $\cB$-measurable set $t$, the return-time function of $t$ is bounded from above, that is, $Y$ is covered by a finite number of iterates of $t$.

Alternatively, the reader may notice that the algebra $\cB$ defined by (*), for which we will apply the lemma, satisfies the additional property right from the start, and thus we could have added this property in the hypotheses of the lemma.
\end{rema*}

The uniform $S$-invariant  algebra $\cA_{0}$ will be generated by a nested sequence of finite partitions $(\alpha_{i}^\infty)_{i\geq 1}$ and their iterates under $S$. Thus any element of $\cA_{0}$ will be obtained as a finite union of iterates of sets of the form
$$
a_{0} \cap  S^{-1}(a_{1}) \cap \cdots  \cap S^{j-1}(a_{j-1})
$$
where $a_{0}, \dots a_{j-1}$ are elements of some finite partition $\alpha_{i}^\infty$. In order to guarantee the uniformity of such sets, we will ensure that the partition $\alpha_{i}^\infty$ is ``$(\varepsilon,i\mbox{-blocks})$ represented by some set $t$'', which will entail that  almost every sufficiently long finite orbit will have ``$\varepsilon$-uniform distribution of $i$-blocks'' of the partition (see below for a precise definition). For every $i$, the partition $\alpha_{i}^\infty$ will be obtained as the limit of a  sequence of partitions $(\alpha_{i}^n)_{ n \geq i}$;  we will need a sub-lemma to construct the partitions $(\alpha_{1}^n, \dots,\alpha_{n}^n)$ by modifying the partitions $(\alpha_{1}^{n-1}, \dots ,\alpha_{n-1}^{n-1})$. 

We will now define precisely the  limit of a sequence of partitions, introduce the notion of ``$\varepsilon$-uniform distribution of $i$-blocks'', and state the sub-lemma we need to construct the partitions $(\alpha_i^j)_{1\leq i\leq j}$.

All partitions are implicitly  supposed to be finite. For us, a partition of $Y$ with exactly $n$ elements is a sequence $\alpha=(a_1,\dots,a_n)$ of mutually disjoint measurable subsets of $Y$, such that $a_1\cup\dots\cup a_n=Y$, and such that each  set $a_i$ has positive measure. We will also see such a partition as a measurable map  $\alpha : Y \ra \{1, \dots , n\}$ such that each set $a_i:=\alpha^{-1}(i)$ has positive measure. This allows us to define the distance between two partitions with $n$ elements as follows.  Consider the set of measurable maps from $Y$ to $\{1, \dots , n\}$, up to zero measure, endowed with the distance
$$
d_{\mathrm{part}}(f,g) = \mu\{x, f(x) \neq g(x)\}.
$$
This is a complete metric space. The space of partitions with exactly $n$ elements (up to zero measure) is an open set of this space. 

\begin{defi}
Let $(Y,\nu,S)$ be a  measured dynamical system and $\alpha$ be a partition of $Y$.  An \emph{$i$-block} of $\alpha$ is  an element $A=(a_{0}, \dots , a_{i-1})$ of $(\alpha)^i$. To every such $i$-block $A$ is associated the \emph{$i$-cylinder} 
$$X_{A,\alpha} = a_{0} \cap  S^{-1}(a_{1}) \cap \cdots  \cap S^{i-1}(a_{i-1}).$$
The \emph{$\alpha$-name} of a (finite) orbit is the sequence of elements of $\alpha$ visited by the orbit.
A finite orbit $y, \dots , S^{N-1}(y)$ has \emph{$\varepsilon$-uniform distribution of $i$-blocks of $\alpha$} if for every $i$-block $A$, the proportion of occurrence of the $i$-block $A$ among the sequence of $(N-i+1)$ $i$-blocks appearing in the $\alpha$-name of the orbit is approximately equal to the measure $\nu(X_{A,\alpha})$ of the cylinder $X_{A,\alpha}$, with error strictly less than  $\varepsilon$.
A partition $\alpha$ is \emph{$(\varepsilon,i\mbox{-blocks})$-represented by a set $t$} if 
 \begin{enumerate}
\item the return-time function $\tau$ of $t$ is bounded from below by $\frac{i}{\varepsilon}$, and bounded from above (by some integer $M$),
\item for all $y \in t$, the orbit $y, \dots , S^{\tau(y)-1}(y)$ has $\varepsilon$-uniform distribution of $i$-blocks of $\alpha$.
\end{enumerate}
\end{defi}

\bigskip
For the next sub-lemma we assume the hypotheses of lemma~\ref{lem.uniform-algebra}. We fix a sequence $(\beta_{i})_{i \geq 1}$ of \emph{nested} partitions by $\cB$-measurable sets, that is, $\beta_{i+1}$ refines $\beta_{i}$ for every $i\geq 1$. Later on we will assume that this sequence generates the algebra $\cB$ (\textit{i.e.} every set in $\cB$ is the union of elements of some partition $\beta_{i}$).

\begin{sublemma}\label{lem.weiss} 
For $n \geq 1$, let  $(\hat \alpha_{i})_{ 1 \leq i \leq n}$ be a sequence of nested partitions, $(t_{i})_{0 \leq i \leq n-1}$ a decreasing sequence of $\cB$-measurable sets, and $(\varepsilon_{i})_{ 1 \leq i \leq n-1}$ a sequence of positive numbers. Assume that  $\hat \alpha_{n}$ refines $\beta_{n}$, and that, for $1 \leq i \leq n-1$,
\begin{itemize}
\item[--] the partition $\hat \alpha_{i}$ refines the partition $\beta_{i}$,
\item[--] the partition   $\hat \alpha_{i}$ is $(\varepsilon_{i}, i\mbox{-blocks})$-represented by  $t_{i}$.
\end{itemize}
Let $\varepsilon_{n} >0$.
Then there exist a $\cB$-measurable  set $t_{n} \subset t_{n-1}$,  and a sequence of nested partitions
$(\alpha_{i})_{1 \leq i \leq n}$ such that for every $1 \leq i \leq n$,
\begin{enumerate}
\item the partition $\alpha_{i}$ refines the partition  $\beta_{i}$,
\item the partition  $\alpha_{i}$ is $(\varepsilon_{i}, i\mbox{-blocks})$-represented by  $t_{i}$,
\item  the partition $\alpha_{i}$ has the same number of elements as $\hat \alpha_{i}$, and $d_{\mathrm{part}}(\alpha_{i},\hat \alpha_{i}) < \varepsilon_{n}$.
\end{enumerate}
\end{sublemma}

\subsection{Proof of sub-lemma~\ref{lem.weiss}}
\label{ss.proof-weiss}

Before entering the proof, let us begin by recalling some urban vocabulary.
A \emph{tower} is a (measurable) set $t$ with positive measure, called the \emph{basis} of the tower, equipped with a (measurable) partition, such that the return-time function $\tau_{t}$ on $t$ is constant on every set of the partition. Given an element $a$ of this partition, and $h=\tau_{t}(a)$, the \emph{column over $a$} is the sequence of sets $(a, S(a), \dots S^{h-1}(a))$. For every point $y \in a$, the sequence of points $(y, S(y), \dots S^{h-1}(y))$ is called a \emph{fibre} of the tower (inside the column over $a$). The \emph{minimal and maximal heights} of the tower are respectively the  minimum and maximum of the function $\tau_{t}$ on $t$.

All our towers will have $\cB$-measurable basis and partition, and since $\cB$ is uniform and has no non-empty null-set their maximal height will be finite (see the remark following lemma~\ref{lem.uniform-algebra}). Thus the partition on $t$ induces a (finite) $\cB$-measurable partition of $Y$ (the elements of this partitions are the sets $S^k(a)$ where $a$ is an element of the partition of $t$ and $k$ is less than the height of the column over $a$).

We will use towers to modify some partition $\hat \alpha$ of $Y$ into a new partition $\alpha$, by the process of \emph{copying and painting} (see~\cite{Glasner}, chapter~15). This consists in selecting some (\emph{source}) fibre of the tower, and copying its $\hat \alpha$-name onto some other (\emph{target})  fibres in the same column to get the new partition $\alpha$. Thus $\alpha$ is the partition  that coincides with $\hat \alpha$ outside the target fibres, and such that the $\alpha$-name of the target fibres is equal to the $\hat \alpha$-name of the source fibre. Note that the elements of the new partition $\alpha$  are in natural one-to-one correspondence with the elements of $\hat \alpha$
(in other words $\alpha$ is well defined as a function).

The following fact, which asserts the existence of towers with large minimal height and bounded maximal height, 
is an adaptation of the classical Rokhlin lemma. 
\begin{fact}[Rokhlin lemma]
We assume the hypotheses of lemma~\ref{lem.uniform-algebra}. Then for every $N>0$, and every $\cB$-measurable positive measure set $t$, there exists a $\cB$-measurable positive measure set $\hat t\subset t$ which is disjoint from its $N$ first iterates, so that the return-time function on $\hat t$ is bounded from below by $N$. Furthermore the return-time function on $\hat t$ is bounded from above.
\end{fact}

\begin{proof}
Using the fact that the algebra $\cB$ has no  atom, we choose a set $t_{0} \subset t$ with positive measure less than $\frac{1}{N}$. Consider $\hat t := t_{0} \setminus (S^{-1}t_{0} \cup \cdots \cup S^{-(N-1)}t_{0})$. This set has positive measure by ergodicity. Since $\cB$ is $S$-invariant, $\hat t$ is $\cB$-measurable, thus uniform. This entails the boundedness of the return-time function on $\hat t$ (see the remark after lemma~\ref{lem.uniform-algebra}).
\end{proof}

From now on we assume the hypotheses of sub-lemma~\ref{lem.weiss}. The proof is divided into three main steps.

\subsubsection*{Step I: choice of a first tower}

\paragraph{I.a. Choice of the height $m'$.---}
The following fact is a direct consequence of Birkhoff theorem (or apply the second item of lemma 15.16 of~\cite{Glasner} to the partition $\alpha\vee\dots\vee S^{-n}\alpha$).
\begin{fact}
For every partition $\alpha$, every $n>0$ and every $\eta >0$ there exists some integer $m'$ such that for every tower of minimal height greater than $m'$, the union of the fibres of the tower having $\eta$-uniform distribution of $n$-blocks of $\alpha$
has measure strictly greater than $1-\eta$.
\end{fact}

Remember that the partition $\hat \alpha_{i}$ is assumed to be $(\varepsilon_{i}, i\mbox{-blocks})$-represented by  $t_{i}$. This property involves a finite number of strict inequalities, thus it is still satisfied if we replace $\varepsilon_{i}$ by a slightly smaller number. We choose a positive real number $\eta$ such that for every $i=1,\dots, n-1$,
the partition $\hat \alpha_{i}$ is  $(\varepsilon_{i}-i\eta, i\mbox{-blocks})$-represented by  $t_{i}$.
We also assume that $\eta$ is much smaller than $\varepsilon_{n}/n$. 

We apply the fact above to this number $\eta$ and the partition $\hat \alpha_{n}$.  We get an integer $m'$  such that for every tower of minimal height greater than $m'$, the union of the fibres of the tower having $\eta$-uniform distribution of $n$-blocks of $\hat\alpha_{n}$ has measure strictly greater than $1-\eta$. We may assume moreover that $m'^{-1}$ is much smaller than $\varepsilon_{n}/n$ (actually ``much'' can be replaced by ``100 times'').

\paragraph{I.b. Choice of the basis $t_n'$.---}
Let $t'_{n} \subset t_{n-1}$ be given by Rokhlin lemma: it is a $\cB$-measurable set such that the return-time function is bounded from below by $m'$, and bounded from above (by some integer). A fibre of a tower with basis $t'_{n}$ is called a \emph{good fibre} (for $\hat \alpha_n$) if  it has $\eta$-uniform distribution of $n$-blocks of $\hat \alpha_{n}$; otherwise it is called a \emph{bad fibre} (for $\hat \alpha_n$). According to the fact stated in I.a the union of the bad fibres has measure less than $\eta$.

\paragraph{I.c. Choice of the partition $\theta$.---}
We consider the tower with basis $t'_{n}$ and we choose a $\cB$-measurable partition of the basis which induces a partition $\theta$ of $Y$ satisfying the following condition.
\begin{itemize}
\item[(A)]  The partition $\theta$ refines the partition $\beta_{n}$ (thus all the fibres of some element of $\theta$ have the same $\beta_{n}$-name).
Furthermore, for every $i=1, \dots, n-1$, the partition $\theta$ refines the partition $(t_{i}, Y \setminus t_{i})$.
\end{itemize}
Note that the partition of $t'_{n}$ by the $\beta_{n}$-names of the fibres is $\cB$-measurable since $t'_{n}$, $S$ and $\beta_{n}$ are $\cB$-measurable. 
The partitions $(t_{i}, Y \setminus t_{i})$ are also $\cB$-measurable. Thus we can find a $\cB$-measurable partition that refines all these partitions, 
and any such partition fulfils our needs.

\subsubsection*{Step II: construction of the partitions $\alpha_{1},\dots,\alpha_n$}

\paragraph{II.a.  General principle of the construction.---}
We will now  modify all the partitions $\hat \alpha_{i}$'s into the new partitions $\alpha_{i}$'s by ``copying and painting" (see the beginning of subsection~\ref{ss.proof-weiss} for the definition of this phrase). Thus we will apply several times the following process: we select some source fibre of the tower $t'_{n}$, and, for each $i = 1, \dots , n$, copy the $\hat \alpha_{i}$-name of this fibre over some target fibres in the same column of the tower. The union of all the target fibres will have measure less than $\eta$.

\paragraph{II.b. Easy verifications.---}
Most properties of the partition $\alpha_{n}$ will follow from the general principle of construction described above (the only exception is the uniformity of $\alpha_n$ which will be obtained below after having chosen the new basis $t_{n}$). More precisely, the general principle of construction described above implies that: (i) the sequence $(\alpha_{1}, \dots, \alpha_{n})$ is still nested; (ii)  for $i = 1, \dots , n$ the partition $\alpha_{i}$ still refines $\beta_{i}$;
(iii) for $i = 1, \dots , n-1$ the partition  $\alpha_{i}$ is still $(\varepsilon_{i}, i-\mbox{blocks})$-represented by the set $t_{i}$; (iv) for $i = 1, \dots , n$  the partition $\alpha_{i}$   is $\varepsilon_{n}$-close to $\hat \alpha_{i}$. 

Let us check this. All the partitions $\hat \alpha_{i}$ are modified simultaneously, by  copying and painting the names of the same sources over the same targets. In particular for every point $x$ there exists some point $y$ such that the sequence $(\alpha_{i}(x))$ coincides with the sequence $(\hat \alpha_{i}(y))$. This entails (i).

The partition $\hat \alpha_{i}$ refines $\beta_{i}$. Since the fibres in a given column have the same $\beta_{n}$-name and thus also the same $\beta_{i}$-name, the partition $\alpha_{i}$ still refines $\beta_{i}$, giving (ii).

The modification of the partition $\hat \alpha_{i}$ takes place on a set of measure less than~$\eta$ and   $\eta< \varepsilon_{n}$. This gives~(iv).

Now let $1 \leq i < n$.  Since $t'_{n} \subset t_{n-1} \subset t_{i}$, every fibre of $t'_{n}$ is the concatenation of some fibres of $t_{i}$.
Furthermore two fibres in the same column of the partition of the $t'_{n}$-tower have the same $(t_{i}, Y \setminus t_{i})$-name, in other words the heights at which they meet the set $t_{i}$ are the same (property~(A) of step I.c).
Thus the modification of $\hat \alpha_{i}$ has consisted in copying the $\hat \alpha_{i}$-name of some whole fibres of $t_{i}$ onto other whole fibres of $t_{i}$. Consequently  the proportion of each $i$-block in the $\alpha_{i}$-name of some fibre of $t_{i}$ is equal to
the proportion of the same $i$-block in the $\hat \alpha_{i}$-name of some (maybe other) fibre of $t_{i}$.
  Furthermore, for every $i$-block $A$, the measures of the $i$-cylinders $X_{A,\alpha_{i}}$ and $X_{A,\hat \alpha_{i}}$ differ by at most $i\eta$. By the choice of $\eta$ at step I.a, the partition $\alpha_{i}$ is still  $(\varepsilon_{i}, i\mbox{-blocks})$-represented by the set $t_{i}$, providing (iii).

\paragraph{II.c. Construction.---}
Here is the precise modification of the partitions $\hat \alpha_{1},\dots,\hat\alpha_n$: in each column, if there exists a good fibre (for $\hat \alpha_n$) in that column, we copy the $\hat \alpha_{i}$-name of this fibre on all the bad fibres of the same column (for every $i=1, \dots, n$). This produces the new partitions $\alpha_{1},\dots,\alpha_n$. Note that this abide by the general principle since the union of all the bad fibres has measure less than~$\eta$.

\subsubsection*{Step III: construction of the set $t_n$}

\paragraph{III.a. Choice of the height $m$.---}
We consider the union $\Delta$ of the fibres of the tower $t_n'$ which are bad for $\alpha_n$: these are the fibres whose $\hat \alpha_n$ and $\alpha_n$-names coincide and are bad for $\widehat\alpha_n$. The set $\Delta$ has measure less than $\eta$ (since it is included in the union of the fibres of $t_n'$ that are bad for $\widehat\alpha_n$).
Moreover by construction this is a union of columns; in particular, $\Delta$ is a $\cB$-measurable set,
hence a uniform set: there exists an integer $m$ such that the set $\Delta$ is $(\eta,m)$-uniform, so that within almost every finite segment of orbit of length greater than $m$, the proportion of points belonging to $\Delta$ is less than $2\eta$. We can (and we will) also assume that $m$ is much bigger than $n$.

Observe that a fibre of $t_n'$ included in $Y\setminus\Delta$ is a ``good fibre" with respect to the new partition $\alpha_n$: more precisely, a fibre $t_n'$ included in $Y\setminus\Delta$ has $\eta$-uniform distribution of the $n$-blocks of $\alpha_n$.

\paragraph{III.b. Choice of the set $t_{n}$.---}
We apply again Rokhlin lemma to get a $\cB$-measurable set $t_{n} \subset t'_{n}$ such that the return-time function in $t_n$ is bounded from above (by some integer), and is bounded from below by $m$.

\paragraph{III.c. Uniformity of the partition $\alpha_{n}$.---}
It remains to check that  the partition $\alpha_{n}$ is $(\varepsilon_{n},  n-\mbox{blocks})$-represented by the set $t_{n}$.
Choose a specific $n$-block $A \in \alpha^n$, and let $b = \nu(X_{A,\alpha_n})$
be the expected proportion of occurrence of the $n$-block $A$.
Consider a fibre $f$ of $t_{n}$. This is a segment of orbit of length $N\geq m$ which is a concatenation of $t'_{n}$-fibres. In order to estimate the number of occurrences of $A$ in the $\alpha_{n}$-name of the fibre $f$,
we categorise the sub-orbits of length $n$ of $f$ into three types:
\begin{itemize}
\item[$(a)$] those that  are included in a fibre of $t_n'$ contained in $Y\setminus \Delta$,
\item[$(b)$] those that  are included in a fibre of $t'_{n}$ contained in $\Delta$,
\item[$(c)$] those that are not included in a fibre of $t'_{n}$.
\end{itemize}
Let $n_a$, $n_b$ and $n_c$ be the number of sub-orbits of length $n$ of type $(a)$, $(b)$ and $(c)$ in the fibre $f$. Then the  number of ``uncontrolled" sub-orbits of length $n$ is $n_0:=n_b+n_c$. We have
$$n_0=n_b+n_c\leq 2\eta N+n\frac{N}{m'}  < N \frac{2\varepsilon_{n}}{100} + N \frac{\varepsilon_{n}}{100} =N \frac{3\varepsilon_{n}}{100}$$
(the estimate of $n_b$ follows from the choice of $m$; the estimate on $n_c$ follows from the fact that the length of a $t'_n$-fibre is at least $m'$; the last inequality follows from the fact that $\eta $ and $m'^{-1}$ have been chosen much smaller, say 100 times smaller, than $\varepsilon_{n}/n$).
Now, using the fact that every fibre of the tower $t_n'$ included  in $Y\setminus\Delta$ has $\eta$-uniform distribution of the $n$-blocks of $\alpha_n$, we obtain that the number of occurrences of $A$ in $f$ is less than 
$$(N-(n-1))(b+\eta) + n_{0} < (N-(n-1))\left( b +\eta + \frac{n_{0}}{N-(n-1)}\right) < (N-(n-1))(b+\epsilon_n)$$
(the last inequality follows from the above inequality on $n_0$, from the fact that $N$ is bigger than $m$ which is much bigger than $n$, and from the fact that $\eta$ is much smaller than $\epsilon_n$).
Similarly, one proves that  the number of occurrences of $A$ in $f$ is more than $(N-(n-1)).(b-\epsilon_n).$
This gives the expected uniformity, and concludes the proof of sub-lemma~\ref{lem.weiss}.

\subsection{Construction of the uniform algebra (proof of lemma~\ref{lem.uniform-algebra})}
\label{ss.algebra}

Since the algebra $\cB$ is countable, we can find  a nested sequence  $(\beta_{i})_{i \geq 1}$  of $\cB$-measurable partitions that generates the algebra $\cB$.  We will construct a nested sequence of  partitions $(\alpha^\infty_{i})_{i \geq 1}$ with the following properties.
\begin{enumerate}
\item $\alpha^\infty_{i}$ refines $\beta_{i}$ for every $i$.
\item The sequence $(\alpha^\infty_{i})_{i \geq 1}$ generates the $\sigma$-algebra of $Y$
modulo a null-set.
\item For every $i \geq 1$, the partition $\alpha_{i}^\infty$ has the following strong uniformity property:
for every $i$-block $A$, the corresponding $i$-cylinder $X_{A, \alpha_{i}^\infty}$ is a uniform set.
\end{enumerate}
\medskip

For each $i\geq 1$, the partition $\alpha_i^\infty$ will be obtain as the limit of a Cauchy sequence of partitions $(\alpha_i^j)_{j\geq i}$. We will now explain how to construct the $\alpha_i^j$'s.
We first fix any summable sequence $(\varepsilon'_{n})_{n \geq 1}$ and
an auxiliary sequence  of nested partitions $(\gamma_{i})_{i \in \NN}$
which  generates the $\sigma$-algebra  of $Y$ and such that  $\gamma_{i}$ refines $\beta_{i}$ for every $i$.
The existence of such a sequence follows from the hypothesis that $Y$, equipped with its $\sigma$-algebra, is a standard Borel space.
We want to construct by induction a triangular array of partitions $(\alpha_{i}^j)_{ 1 \leq i \leq j}$
$$
\begin{array}{cccccc}
\alpha_{1}^1  &   \\
\alpha_{1}^2  & \alpha_{2}^2   &    \\
\vdots    &  \vdots  &   \ddots   &     \\
\alpha_{1}^{n-1}  & \alpha_{2}^{n-1} &  \cdots &  \alpha_{n-1}^{n-1}   &  \\
\vdots & \vdots && \vdots & \ddots
\end{array}
$$
such that 
\begin{enumerate}
\item[i.] the horizontal sequence $(\alpha^j_{1}, \dots , \alpha^j_{j})$ is nested for every $j\geq 1$;
\item[ii.] the partitions in $\alpha_i^i,\alpha_i^{i+1},\alpha_i^{i+2},\dots$ in the $i^{th}$ column refine $\beta_{i}$, have the same number of elements, and satisfy $d_{\mathrm{part}}(\alpha^{j-1}_{i},\alpha^{j}_{i}) < \varepsilon'_{j}$ for every $i \leq j$;
\item[iii.] $d_{\mathrm{part}}(\alpha^{j}_{j},\hat\gamma_{j}) < \varepsilon'_{j}$ where $\hat\gamma_j$ is some partition refining $\gamma_j$ for every every $j\geq 1$;
\item[iv.] there exists a decreasing  sequence $(t_{i})_{i\geq 1}$ of $\cB$-measurable subsets of $Y$
such that for every $i \leq j$ the partition $\alpha_{i}^j$ is ($\varepsilon_{i}$,$i$-blocks)-represented by $t_{i}$ where $\varepsilon_{i} \leq \frac{\varepsilon'_{i}}{(\#\alpha_{i}^i)^{i}}=\frac{\varepsilon'_{i}}{(\#\alpha_{i}^j)^{i}}$.
\end{enumerate}
To construct this triangular array, we first apply sub-lemma~\ref{lem.weiss} with $n=1$, $t_{0} = Y$, $\hat \alpha_{1} = \gamma_{1}$, and $\varepsilon_{1}=\frac{\varepsilon'_{1}}{\#\alpha_{1}^1}$. We get a set $t_{1}$ and a partition $\alpha_{1}^1= \alpha_{1}$  still refining $\beta_{1}$, which is  ($\varepsilon_{1}$, $1$-blocks)-represented by $t_{1}$, and $\varepsilon_{1}$-close to $\gamma_{1}$. Now, suppose that, for some $n \geq 2$, the partitions $(\alpha_{i}^j)_{ 1 \leq i\leq j \leq n-1}$ (that is, the $n-1$ first rows of the triangular array above), the partitions $(\hat \gamma_j)_{1\leq j\leq n-1}$ and the sets $(t_i)_{1\leq i\leq n-1}$ have been constructed and satisfy properties (i)...(iv). Let us construct the $n^{th}$ row of the array. For this we apply sub-lemma~\ref{lem.weiss}  with the following data:
\begin{itemize}
\item[--] $\hat \alpha_{i} := \alpha_{i}^{n-1}$ for $i=1, \dots , n-1$;
\item[--] $\hat \alpha_{n} := \hat \gamma_{n}$ is any partition which refines both $\hat \alpha_{n-1} = \alpha_{n-1}^{n-1}$, $\beta_{n}$  and $\gamma_{n}$;
\item[--] the sets $t_{i}$ are as above;
\item[--] $\varepsilon_{n} \leq \frac{\varepsilon'_{n}}{(\#\hat \alpha_{n})^{n}}$.
\end{itemize}
The sub-lemma provides a set $t_{n}$  and a sequence of partitions $\alpha_{1}, \dots, \alpha_{n}$. Let us set $\alpha_{i}^n := \alpha_{i}$. Then the conclusion of the sub-lemma shows that this sequence fits as the new horizontal sequence in our triangular array, that is, properties (i)...(iv) above are still satisfied at order $n$.
Thus by induction  we get the existence of an infinite triangular array of partitions $(\alpha_{i}^j)_{ 1 \leq i \leq j}$ satisfying properties (i)...(iv).

As the sequence $(\varepsilon_{n})$ is summable, for each fixed $i$, the  (vertical) sequence $(\alpha_{i}^j)_{j \geq i}$
is a Cauchy sequence for the distance $d_{\mathrm{part}}$: it
is converging towards some function \hbox{$\alpha_{i}^\infty : Y\rightarrow\{1, \dots ,n \}$.} As the convergence is as fast as we wish (at each step the number $\varepsilon_{n}$ is chosen as little as we wish), we may assume this function is a partition with  the same number of elements as $\hat \alpha_{i}$ (remember that each element is supposed to have positive measure).
Properties (i) and (ii) imply that up to modifying each partition $\alpha_i^\infty$
on a set having zero measure, one can assume that the sequence $(\alpha_i^\infty)$ is nested
and that $\alpha_i^\infty$ refines $\beta_i$ for every $i$, giving property 1 above.

\medskip

By applying the following fact to $\hat \alpha_n:=\gamma_n$ and
$\alpha_n:=\alpha_n^\infty$, one gets property 2.

\begin{fact}
Let $(Y, \cA, \nu)$ be a measured space. Let $(\hat \alpha_{n})_{n\geq 1}$ and 
$(\alpha_{n})_{n\geq 1}$ be two nested sequences of partitions such that the sequence 
$(d_{\mathrm{part}}(\hat \alpha_{n},\alpha_{n}))_{n\geq 1}$ tends to $0$.
Assume $(\hat \alpha_{n})_{n\geq 1}$ generates the $\sigma$-algebra $\cA$.
Then $(\alpha_{n})_{n\geq 1}$ generates the $\sigma$-algebra $\cA$ modulo a null-set.
\end{fact}

\begin{proof}
We first choose some $\hat a \in \hat \alpha_{n_{0}}$ for some $n_{0}$ and prove that there exists a set $a$ in the $\sigma$-algebra generated  by $(\alpha_{n})$ such that the symmetric difference $\hat a \triangle a$ is a null-set.
For every $k$ we can find within the algebra generated by some $\alpha_{n}$ (with large $n$) a set $a'_{k}$ such that 
$$
\nu(a'_{k} \triangle \hat a) < \frac{1}{2^k}. 
$$
For a fixed $k_{0}$, the set 
$$
\cup_{k > k_{0}} a'_{k}
$$
contains $a$ modulo a null-set, and the measure of the difference is less than $1/2^{k_{0}}$.
Thus the intersection $a$ of these sets (for all $k_{0}$) is equal to $a$ modulo a null-set, as required.

Now let $(\hat a_{\ell})_{\ell \geq 1}$ be an enumeration of all the elements of all the partitions $\hat \alpha_{n}$: this sequence generates the $\sigma$-algebra $\cA$.
For each $\ell$ let $a_{\ell}$ be a set in the $\sigma$-algebra generated by $(\alpha_{n})$ which is equal to $\hat a_{\ell}$ modulo a null-set.
Let 
$$
Y_{0}:=Y\setminus \bigcup_{\ell \geq 0} \left(\hat a_{\ell} \triangle a_{\ell}\right).
$$
Then $Y_{0}$ is a full measure set, and $({\alpha_n}_{\mid Y_0})$ generates $\cA_{\mid Y_{0}}$. Thus $(\alpha_{n})$ generates $\cA$ modulo a null-set.
\end{proof}

\medskip

Let us now check that $\alpha_{i}^\infty$ satisfies property 3 (the strong uniformity). We start with a preliminary remark.

\medskip

\noindent{\bf General remark.}
If an orbit has $\varepsilon$-uniform distribution of $k$-blocks of a partition $\alpha$, and if $\alpha$ refines a partition $\alpha'$, then the same orbit has $(\# \alpha)^k.\varepsilon$-uniform distribution of $i$-blocks of $\alpha'$ for any $i\leq k$. Indeed, $(\# \alpha)^k$ is the number of $k$-blocks of $\alpha$ and every $i$-cylinder $\alpha'$ is the disjoint union of some $k$-cylinders of $\alpha$.
In particular if $\alpha$ is $(\varepsilon,k\mbox{-blocks})$-represented by $t$ then $\alpha'$ is $( (\# \alpha)^k.\varepsilon,i\mbox{-blocks})$-represented by $t$.

\medskip

\noindent For every $j \geq k \geq i$, the partition $\alpha_{i}^j$ is refined by the partition $\alpha_{k}^j$, which is ($\varepsilon_{k}$,$k$-blocks)-represented by the set $t_{k}$ (by property (iv) above).
By definition of $\varepsilon_k$, one has
$$
\varepsilon'_{k} \geq (\#\alpha_{k}^k)^{k}.\varepsilon_{k}  = (\#\alpha_{k}^j)^{k}.\varepsilon_{k}.
$$
Applying the former remark, we get that the partition $\alpha_{i}^j$ is ($\varepsilon'_{k}$,$i$-blocks)-represented by $t_{k}$. By definition the fibres of $t_{k}$ have $\varepsilon'_{k}$-uniform distribution of $i$-blocks of $\alpha_{i}^j$; the following fact tells us that this is still true for every sufficiently long orbits, up to a factor $4$.

\begin{fact}
Let $N(t_{k})$ be the upper bound of the return-time function on the set $t_{k}$ (i.e. the height of the tower with basis $t_k$). Then for almost every\footnote{The fact is actually  true for \emph{every} $y\in Y$, but this property will disappear when we will let $i,j,k$ go to infinity.} $y \in Y$, for every $n \geq \frac{N(t_{k})}{\varepsilon'_{k}}$, the orbit $y, \dots , S^{n-1}(y)$ has $4\varepsilon'_{k}$-uniform distribution of $i$-blocks of~$\alpha_{i}^j$.
\end{fact}

\begin{proof}
Let $y \in Y$. Using the upper bound $N(t_{k})$ on the return-time, we can ignore the beginning of the orbit of $y$ until it first meets the set $t_{k}$, 
and the end after it last leaves $t_{k}$, increasing the error at most by $2\varepsilon'_{k}$.
Using the lower bound (the return-time function of $t_{k}$ is greater than $\frac{k}{\varepsilon'_{k}}$), we can ignore the $i$-blocks that are not included in some fibre of the $t_{k}$-tower, 
increasing the error again at most by $\varepsilon'_{k}$.
The proportion of remaining $i$-blocks is estimated up to $\varepsilon'_{k}$ since $t_{k}$ represents the partition $\alpha_{i}^j$.
\end{proof}

\noindent
For every $i\geq 1$, using a diagonal process, it is easy to construct a full measure set $Y' \subset Y$ and an extracted sequence $(\alpha_{i}^{j_{\ell}})_{\ell\geq 1}$ such that and every $y \in Y'$, the sequence $(\alpha_{i}^{j_{\ell}}(y))$ is stationary. Up to diminishing it by a null-set, we can assume $Y'$ is invariant under $S$. Thus we see that the fact above still holds when we replace the partition $\alpha_{i}^j$ by $\alpha_{i}^\infty$.  This gives the required strong uniformity property for $\alpha_i^\infty$ (property~3). This concludes the construction of the partitions $(\alpha_i^\infty)_{i\geq 1}$.

\medskip

Let $\cA_{0}$ be the algebra generated by the sequence $(\alpha_{i}^\infty)$ and the dynamics:
any element of $\cA_{0}$ is the union of iterates of $i$-cylinders
of some partition $\alpha_{i}^\infty$.
This algebra is countable, $S$-invariant and has no null-set. It contains $\cB$ by property 1
and in particular has no atom.
It generates the $\sigma$-algebra of $Y$ modulo a null-set, according to property 2.
Let us check the uniformity of any element $a\in\cA_{0}$:
observe that $a$ is a finite disjoint union of iterates of $i$-cylinders of
$\alpha_{i}^\infty$ for some $i$.
By property 3, the $i$-cylinders of $\alpha_{i}^\infty$ are uniform.
Since uniformity is invariant under $S$ and preserved by finite disjoint unions,
$a$ is uniform.
This completes the proof of lemma~\ref{lem.uniform-algebra}.

\subsection{Proof of the relative Jewett-Krieger theorem~\ref{theo.weiss}}
\label{ss.proof-Weiss}

We now complete the proof of theorem~\ref{theo.weiss}: using the algebra $\cA_{0}$, we construct a uniquely ergodic minimal homeomorphism on a Cantor set $(C,\nu',g)$ isomorphic to our initial system $(Y,\nu,S)$. This construction is quite standard: it uses the coding of the dynamics by the shift and  classical results about standard Borel spaces. Alternatively, the authors of~\cite{Hansel-Raoult} use~\emph{Stone duality} at this point of the construction.

\medskip

We assume the hypotheses of theorem~\ref{theo.weiss}, and denote by $\cB$ the algebra consisting of the preimages under $p$ of the clopen sets in $\cK$. We have already observed that this algebra satisfies the hypotheses of lemma~\ref{lem.uniform-algebra}. We apply this lemma and obtain a new algebra~$\cA_0$ which contains $\cB$, is countable, has no atom and no non-empty null-set, is made of uniform sets, and generates the $\sigma$-algebra of $Y$ modulo a null-set.

We first restrict the space $Y$ to a full measure invariant subset\footnote{The set $Y_{0}$, equipped with the induced $\sigma$-algebra, is still a standard Borel space, see~\cite{Kechris}; but we will not use this fact.} $Y_{0}$, defined as follow.
Let $Y'$ be a full measure subset of $Y$, such that the $\sigma$-algebra generated by the algebra $\cA_{0 \mid Y'}$
coincides with the $\sigma$-algebra of $Y$ induced on $Y'$:
this implies that the algebra $\cA_{0 \mid Y'}$ separates the points of $Y'$.
By taking the intersection of all the iterates of this set $Y'$, we get a set $Y_{1}$ with the same property and which is in addition $S$-invariant.
Let $Y_{2}$ be the complementary set of the union of all the null-sets in $\cA_{0}$; this is another full measure $S$-invariant set. The set $Y_{0}=Y_{1} \cap Y_{2}$ is a full measure $S$-invariant subset of $Y$, whose pairs of points are separated by the elements of  $\cA_{1}:=\cA_{0 \mid Y_{0}}$. 
Note that $\cA_{1}$ is obviously an algebra of $Y_{0}$, and that each element in $\cA_{1}$, being the intersection of a uniform set with the full-measure invariant set $Y_{0}$, is still a uniform set.
We denote by $\cA_{1}^*$ the collection of non-empty elements of this algebra; note that these sets have positive measure.

\subsubsection*{Definition of $C$, $g$ and $\nu'$}

We consider the Cantor set $\{0,1\}^{\cA_1^*}$ equipped with the usual product topology. For $A\in\cA_1^*$, we denote by $p_A:\{0,1\}^{\cA_1^*}\to \{0,1\}$ the projection on the $A$-coordinate.  We consider the natural coding map $\Phi:Y_{0} \to \{0,1\}^{\cA_1^*}$ defined by ($p_A(\Phi(y))=1 \Leftrightarrow y \in A$). It has the following properties.
\begin{itemize}
\item[--] The map $\Phi$ is one-to-one, since the pairs of points in $Y_0$ are separated by $\cA_1^*$.
\item[--] The map $\Phi$ is the restriction of the map $Y \to \{0,1\}^{\cA_1^*}$ defined by the same formulae, and this map is measurable: indeed the Borel $\sigma$-algebra of  $\{0,1\}^{\cA_1^*}$ is generated by sets of the form $p_{A}^{-1}(1)$, and the preimage of such a set is the set $A$, which belongs to the $\sigma$-algebra of $Y$.
\end{itemize}
Remember that the set $Y$, equipped with its $\sigma$-algebra, is a standard Borel space.
Thus we can apply corollary 15.2 of~\cite{Kechris}, which says that any injective image of a Borel set under a Borel map is again a Borel set,
getting that $\Phi(Y_{0})$ is
a Borel set in $\{0,1\}^{\cA_1^*}$. Furthermore we can define the measure  $\nu':=\Phi_*\nu$ on $\{0,1\}^{\cA_1^*}$, and $\Phi$ is an isomorphism between 
$(Y_{0},\nu)$ and $(\Phi(Y_{0}),\nu')$, both equipped with their Borel $\sigma$-algebra.

We define the set $\cC$ as the closure in $\{0,1\}^{\cA_1^*}$ of the set $\Phi(Y_0)$.
We define an \emph{elementary cylinder in $\cC$} to be a set of the form $\cC\cap p_{A}^{-1}(1)$ for some $A\in \cA_{1}$.
The following property will be used several times.
\begin{fact}
Every elementary cylinder $\cC\cap p_{A}^{-1}(1)$ is equal to $\mbox{Cl}(\Phi(A))$.
Furthermore the family of elementary cylinders forms a basis of the (relative) topology on $\cC$.
\end{fact}
\begin{proof}
The first sentence is easy. Now the definition of $\Phi$ yields the equalities
$$
\Phi(Y_{0}) \cap p_{A_{1}}^{-1}(0)= \Phi(Y_{0}) \cap p_{Y_{0}\setminus A_{1}}^{-1}(1)
$$
$$
\Phi(Y_{0})  \cap \left( p_{A_{1}}^{-1}(1)\cap\dots\cap p^{-1}_{A_{k}}(1) \right) = \Phi(Y_{0})  \cap p_{A_{1} \cap \dots \cap A_{k}}^{-1}(1)
$$
for any  $A_{1},\dots,A_{k} \in \cA_{1}$. Since sets of the form  $p^{-1}(A)$ are clopen these equalities remains valid when $\Phi(Y_{0})$ is replaced by its closure, which is equal to $\cC$. This gives the claim.
\end{proof}

Since the elements of $\cA_{1}^*$ have positive measure, the elementary cylinders have positive measure
and the measure $\nu'$ has full support in $\cC$ (every nonempty open set has positive measure).
Since the algebra $\cA_1$ is $S$-invariant, there is a natural shift map \hbox{$\sigma:\{0,1\}^{\cA_1^*}\to \{0,1\}^{\cA_1^*}$} defined by $p_{S(A)}\circ\sigma=p_A$.  The set $\cC\subset \{0,1\}^{\cA_1^*}$ is invariant under the shift map $\sigma$; we define $g:=\sigma_{|\cC}$. Clearly, $g$ is a homeomorphism of $\cC$, and one has $\Phi\circ S=g\circ \Phi$.

\subsubsection*{Definition of the semi-conjugacy $\Pi$}
Recall that $\cB$ is the preimage under the projection $p:Y\to \cK$ of the algebra of the clopen sets in $\cK$. Let $\cB^*:=\{A\in\cB \mbox{ such that }\nu(A)>0\}=\cB\setminus\{\emptyset\}$. 
We consider the Cantor set $\{0,1\}^{\cB^*}$, the natural coding map $\Psi:\cK\to \{0,1\}^{\cB^*}$ defined by $p_B(\Psi(x))=1 \Leftrightarrow x \in p(B)$, and the shift map $\sigma:\{0,1\}^{\cB^*}\to\{0,1\}^{\cB^*}$. The map $\Psi$ is a homeomorphism on its image, and provides a conjugacy between $f : \cK \to \cK$ and the restriction of  $\sigma$ on $\Psi(\cK)$.

We may identify $\cB$ with some subset of $\cA_{1}$ by the one-to-one map $i:b \mapsto b \cap Y_{0}$.
Thus there is a natural projection of the Cantor set $\{0,1\}^{\cA_1^*}$ onto the  Cantor set $\{0,1\}^{\cB^*}$, obtained by forgetting the coordinates corresponding to the elements of $\cA_1^*\setminus i(\cB^*)$. Moreover, the image of $\cC$ under this natural projection is contained in $\Psi(\cK)$. We define the projection $\Pi : \cC \to \cK$ as the composition of the natural projection $\cC \subset \{0,1\}^{\cA_1^*}  \to \{0,1\}^{\cB^*}$ with the identification $\Psi^{-1}: \Psi(\cK)\subset \{0,1\}^{\cB^*}\to\cK$. Note that on $Y_{0}$ we have $\Phi S = g \Phi$ and $p= \Pi \Phi$, and on $\cC$ we have $\Pi g = f \Pi$, that is, the diagram of theorem~\ref{theo.weiss} commutes almost everywhere.

\subsubsection*{Unique ergodicity and minimality}
Let us prove that $g$ is uniquely ergodic. Since the elementary cylinders form a basis of the topology on $\cC$, it suffices to prove the following fact.
\begin{fact}
The Birkhoff averages $(\frac 1 n S_{n}(1_{A'},c))_{n \geq 1}$ converge towards $\nu'(A')$, for any elementary cylinder $A'$ of the Cantor set $\cC$ and any point $c\in \cC$. 
\end{fact}

\begin{proof}
Consider an elementary cylinder $A'=\cC \cap p_{A}^{-1}(1)$ in $\cC$ and a point $c\in \cC$. Let $\epsilon>0$.
The set $A= \Phi^{-1}(A')$ is uniform since it belongs to $\cA_{1}^*$.
So there exists $N$ such that, for every $n\geq N$, for $\nu$-a.e. $y \in Y$,
$$~  \quad \quad \quad\quad  \quad\quad\quad \quad\quad \quad\quad \quad \quad\left|\frac 1 n S_n(1_A,y)-\nu(A)\right|\leq \epsilon. \quad \quad \quad \quad \quad\quad \quad \quad\quad \quad\quad \quad \quad  (*)$$
Let us fix any $n \geq N$. Since $A'$ is clopen in $\cC$,  there exists a neighbourhood $U$ of $c$ in $\cC$ such that any point $c'\in U$ satisfies $S_n(1_{A'},c')=S_n(1_{A'},c)$.
Since $\nu'=\Phi_{*}\nu$ has full support in $\cC$,
there exists a point $y\in Y_{0}$ that satisfies $(*)$ and whose image 
$c'=\Phi(y)$ belongs to $U$.
By conjugacy,
$S_n(1_A,y)=S_n(1_{A' },c')$ and $\nu(A)=\nu'(A')$.
Henceforth
$$\left|\frac 1 n S_n(1_{A'},c)-\nu'(A')\right|\leq\epsilon$$
which proves the fact and the unique ergodicity of $g$.
\end{proof}

Since  $g$ is uniquely ergodic and every non-empty open subset of $C$ has positive $\nu'$-measure, the map $g$ is also minimal. Finally,  the topological space $C$ is a Cantor set: indeed it is a closed subset of a Cantor set and it has no isolated point since $g$ is minimal. This completes the proof of theorem~\ref{theo.weiss}.


\section{Open version}

\subsection{Statement and strategy}\label{ss.statement-open}

\begin{theo}
\label{theo.open}
In the conclusion of theorem~\ref{theo.weiss} one can add that the map $\Pi$ is open.
\end{theo}

As for theorem~\ref{theo.weiss}, the proof of theorem~\ref{theo.open} essentially amounts to the construction of a $S$-invariant algebra $\cA_0$ of uniform subsets of $Y$.  Of course, we need an extra property on $\cA_0$ to ensure that the map $\Pi$ is open. 

\begin{defi}
Let $\cB$ be an algebra of measurable subsets of some measured space $(Y,\nu)$.
A set $X \subset Y$ is said to \emph{fill up a $\cB$-measurable set $B$} if $\nu(X\cap B')>0$ for every non-empty $\cB$-measurable set $B'\subset B$. A set $X\subset Y$ is said to be \emph{full} if there exists $B\in\cB$ such that $X\subset B$ modulo a null-set (that is, $\nu(X \setminus B)=0$)
 and such that $X\cap B$ fills up $B$. 
\end{defi} 
 
\begin{lemma}
\label{lem.uniform-algebra-open}
In the conclusion of lemma~\ref{lem.uniform-algebra} one can add that every element of the algebra~$\cA_0$ is full.  
\end{lemma} 
 
As in the preceding section, the algebra $\cA_0$ will be constructed generated by a infinite sequence $(\alpha_i^\infty)_{i\geq 1}$ of finite partitions and the dynamics $S$. In order to get lemma~\ref{lem.uniform-algebra-open}, we will need that, for every $i\geq 1$, every element of the partition $\alpha_i^\infty\wedge\dots\wedge S^{-(i-1)}(\alpha_i^\infty)$ is full (we will say that ``every $i$-cylinder of the partition $\alpha_i^\infty$ is full"). As in the preceeding section, for every $i\geq 1$, the partition $\alpha_i^\infty$ will be obtained as the limit of a convergent sequence of partitions $(\alpha_i^j)_{j\geq i}$, and, for every $j\geq i\geq 1$, the partition $\alpha_i^{j+1}$ will be obtained by modifying the partition $\alpha_i^j$ on a small set. Of course, we will need the $i$-cylinders of the partition $\alpha_i^j$ to satisfy a ``fullness property" for every $j\geq i\geq 1$. A technical issue arises here: the ``fullness property" we need for the $i$-cylinders 
 of the partition $\alpha_i^j$ is not stable under small perturbation of the partition $\alpha_i^j$. This will lead us to introduce another property that we call \emph{stability under bracket}. Let us make this precise. 

\begin{defi}
Let $\alpha,\beta$ be two partitions, such that $\beta$ is $\cB$-measurable. We will say that \emph{the $i$-cylinders of $\alpha$ are $\beta$-full} if for every $i$-block $A$ of $\alpha$, the cylinder $X_{A,\alpha}$ fills up every element of $\beta$ that it meets.
\end{defi}

Given two $i$-blocks $A=(a_0,\dots,a_{i-1})$ and $A'=(a'_0,\dots,a'_{i-1})$ of $\alpha$ and an integer $\ell\in\{1,\dots,i-1\}$, we denote by   $[A,A']_{\ell}$ the $i$-block obtained by concatenating the $\ell$ first elements of $A$ and the $i-\ell$ last elements of $A'$, that is 
$$[A,A' ]_{\ell} = (a_0, \dots, a_{\ell-1} , a'_{\ell} , \dots , a'_{i-1}) .$$

\begin{defi}
Let $\beta$ be a $\cB$-measurable partition of $Y$, and $t$ be a subset of $Y$.  We say that \emph{the $i$-blocks of a partition $\alpha$ are ($\beta,t$)-stable under brackets} if the following holds:  for every set $B \in \beta$, every integer $\ell\in\{1,\dots,i-1\}$,
\begin{itemize}
\item[--]  the set $S^\ell(B)$ is either disjoint from or contained in  $t$,
\item[--] if  $S^\ell(B)$ meets (and thus is included in) $t$, then for every $i$-blocks $A,A'$ of $\alpha$ such that the cylinders $X_{A,\alpha}$ and $X_{A',\alpha}$ both meet $B$, the cylinder $X_{[A,A']_\ell,\alpha}$ also meets $B$.
\end{itemize}
\end{defi}

It is clear that if the $i$-cylinders of $\alpha$ are $\beta$-full, then they are also $\beta'$-full for any partition $\beta'$ that refines $\beta$. The bracket stability has some more subtle monotonicity property.

\begin{fact}[monotonicity of bracket stability]
Consider two $\cB$-measurable partitions $\beta,\beta'$ such that $\beta'$   refines $\beta$,  and two  $\cB$-measurable sets $t' \subset t$.
Assume $\beta'$ and $t'$  satisfy the first property in the definition of bracket stability: the partitions $S^\ell(\beta'), \ell = 1, \dots, i-1$ refines the partition $(t',Y \setminus t')$.
Assume  the $i$-cylinders of some partition $\alpha$ are $\beta$-full and ($\beta,t$)-stable under brackets.
  Then they are also  ($\beta',t'$)-stable under brackets.
\end{fact}

The proof of the fact is left to the reader.
Now we state a new version of sub-lemma~\ref{lem.weiss}.
As before we assume the hypotheses of lemma~\ref{lem.uniform-algebra}, and we  fix a sequence of $\cB$-measurable nested partitions $(\beta_{i})_{i \geq 1}$ which  generates the algebra $\cB$.
\begin{sublemma}\label{lem.open} 
For $n \geq 1$, let $(\hat \alpha_{i})_{ 1 \leq i \leq n}$ be a sequence of nested partitions, $(t_{i})_{0 \leq i \leq n-1}$ a decreasing sequence of $\cB$-measurable sets, $(\varepsilon_{i})_{ 1 \leq i \leq n-1}$ a sequence of positive numbers, and $(N_{i})_{1 \leq i \leq n-1}$ an increasing  sequence of integers. Assume that  $\hat \alpha_{n}$ refines $\beta_{n}$ and that for $1 \leq i \leq n-1$,
\begin{itemize}
\item[--] the partition $\hat \alpha_{i}$ refines the partition $\beta_{i}$,
\item[--] the partition   $\hat \alpha_{i}$ is $(\varepsilon_{i}, i\mbox{-blocks})$-represented by  $t_{i}$,
\item[--]  the $i$-cylinders of  $\hat \alpha_{i}$ are $\beta_{N_{i}}$-full and ($\beta_{N_{i}}, t_{i}$)-stable under brackets.
\end{itemize}
Let $\varepsilon_{n} >0$.
Then there exist a sequence of nested partitions $(\alpha_{i})_{1 \leq i \leq n}$, a $\cB$-measurable set $t_{n}\subset t_{n-1}$ and an integer $N_n\geq N_{n-1}$ such that, for every $1 \leq i \leq n$,
\begin{enumerate}
\item the partition $\alpha_{i}$ refines the partition  $\beta_{i}$,
\item the partition  $\alpha_{i}$ is $(\varepsilon_{i}, i\mbox{-blocks})$-represented by  $t_{i}$,
\item the $i$-cylinders of  $\alpha_{i}$ are $\beta_{N_{i}}$-full and ($\beta_{N_{i}}, t_{i}$)-stable under brackets,
\item  the partition $\alpha_{i}$ has the same number of elements as $\hat \alpha_{i}$, and $d_{\mathrm{part}}(\alpha_{i},\hat \alpha_{i}) < \varepsilon_{n}$.
\end{enumerate}
\end{sublemma}

\subsection{Proof of sub-lemma~\ref{lem.open}}
We explain how to modify the proof of sub-lemma~\ref{lem.weiss} to get sub-lemma~\ref{lem.open}

\subsubsection*{Step I: choice of a first tower}

\paragraph{I.a. Choice of the height $m'$.---}
As before.

\paragraph{I.b. Choice of the basis $t_n'$.---}
As before.

\paragraph{I.c. Choice of the partition $\theta$.---}
We consider a tower with basis $t'_{n}$, and we choose a $\cB$-measurable partition of the basis which induces a partition $\theta$ of $Y$ satisfying  the same condition~(A) as in the proof of sub-lemma~\ref{lem.weiss} ($\theta$ refines $\beta_{n}$), as well as the following condition.
\begin{itemize}
\item[(B)] For every $x,x'$ belonging to the same element of  $\theta$, for every $\ell \in \{0, \dots , n-1\}$, the fibres of $S^\ell (x)$ and $S^\ell (x')$ have the same height, and the same $\beta_{N_{n-1}}$-name.
\end{itemize}
Note that as a consequence of condition~(B), for every $x,x'$ belonging to the same element of  $\theta$ and every $\ell \in \{0, \dots , n-1\}$, the points $S^\ell (x)$ and  $S^\ell (x')$ are at the same level of the $t'_{n}$-tower, and have the same return-time in the basis. Conditions (A) and (B) amounts to saying that the partition $\theta$ refines some $\cB$-measurable partition of $Y$, thus they are achieved by choosing a sufficiently fine partition of the basis.

\subsubsection*{Step II: construction of the partitions $\alpha_{1},\dots,\alpha_n$}

\paragraph{II.a.  General principle of the construction.---}
As before.

\paragraph{II.b. Easy verifications.---}
As before. After the explicit construction, it will remain to check the uniformity of $\alpha_{n}$ (property 2 for $i=n$), and property 3 (fullness and bracket stability).

\paragraph{II.c. Construction.---}
In what follows, we will only mention what happens to the partition $\hat \alpha_{n}$; according to the general principle of the construction, the modifications of the partitions $\hat \alpha_{1},\dots,\hat\alpha_{n-1}$  follow from the modifications of $\hat\alpha_n$. Let us explain the outline of the construction.
As before we will make a (first)  modification of the partition $\hat \alpha_{n}$ copying good fibres over  bad fibres; the key difference, as compared to the proof of sub-lemma~\ref{lem.weiss}, is that we will choose some very small sets $Z_{1},Z_{2}$ that  will be kept unmodified, together with their first $n$ iterates. On the set $Z_{1}$ we will not make any modification whatsoever, and this will preserve the fullness and bracket stability of the partitions $\alpha_{i}$ for $i<n$.
Then the set $Z_{2}$ will be painted so as to get the fullness and bracket stability for the partition $\alpha_{n}$.

We first choose an integer $L$  much bigger than  $\max(\tau_{t'_{n}})/\varepsilon_{n}$,
where $\tau_{t'_{n}}$ is the return-time function of $t'_{n}$. For every set $Z$ we define the set $T(Z)$ as the union of the fibres of the tower $t'_{n}$ that meet $S^\ell(Z)$ for some $\ell\in\{0,\dots,n-1\}$.
\footnote{Here, for every point $z$, the set $T(\{z\})$ is the union of at most two fibres; the exposition is designed to work almost verbatim in the bi-ergodic open case in section~\ref{sec.bi-ergodic-open} below.}
We pick two measurable sets $Z_{1}$ and $Z_{2}$ with the following properties.
\begin{enumerate}
\item $Z_{1}$ and $Z_{2}$ are disjoint and $Z_{1} \cup Z_{2}$ is disjoint from its first $L$ iterates.
\item\label{Z.2} The union of the fibres of $t'_n$ that are bad,
or that are contained in  $T(Z_1\cup Z_2)$,
has measure less than $\eta$.
\item\label{Z.3} For every $i = 1, \dots , n-1$, for every $i$-block $A$ of the partition $\hat \alpha_{i}$, for every set $B \in \beta_{N_{i}}$, if the $i$-cylinder $X_{A,\hat\alpha_i}$ meets (and thus fills up) $B$, then $X_{A,\hat\alpha_i} \cap Z_{1}$ still fills up $B$.
\item  $Z_{2}$ fills up $Y$.
\end{enumerate}
The sets $Z_1,Z_2$ can be obtained as follow. For each triple $(i,A,B)$ as in property~\ref{Z.3} above ($i \in\{ 1, \dots , n-1\}$, $A$ an $i$-block  of the partition $\hat \alpha_{i}$ and $B \in \beta_{N_{i}}$)
and each non-empty $\cB$-measurable subset $B'$ of $B$, one considers a subset $Z(i,A,B,B')\subset X_{A,\hat \alpha_i}\cap B'$ having a small positive measure. Note that there are countably many such triples, so up to reducing each set $Z(i,A,B,B')$ one can assume that $f^k(Z(i_1,A_1,B_1,B'_1))$ and $f^\ell(Z(i_2,A_2,B_2,B'_2))$ for $k,\ell\in \{0,\dots,L\}$ are disjoint if $(i_1,A_1,B_1,B'_1,k)$ and $(i_2,A_2,B_2,B'_2,\ell)$ are distinct. One may decompose $Z(i,A,B,B')$ as two disjoint sets $Z_1(i,A,B,B')\cup Z_2(i,A,B,B')$ with positive measure. Then $Z_1$ is the union of all the sets $Z_1(i,A,B,B')$ and $Z_2$ is the union of all the sets $Z_2(i,A,B,B')$.

We now make a first modification. In each column, if there exists a good fibre in that column, we copy its $\hat \alpha_{n}$-name on all the bad fibres of the same column, except for the bad fibres included in $T(Z_{1} \cup Z_{2})$. This produces a new partition, which we denote by $\gamma_{1}$.
We then choose $N'_{n}$ so that $S^\ell(\beta_{N'_{n}})$
refines the partition of $Y$ associated to the tower $t'_n$ for each $\ell\in \{0,\dots,n\}$.
Note in particular that if $S^\ell(B)$ meets $t'_{n}$
for some $B \in \beta_{N'_{n}}$ and $\ell \in \{1,\dots, n-1\}$,
then $S^\ell(B)$ is included in $t'_{n}$: in other words $\beta_{N'_{n}}$ and $t'_{n}$ satisfy the first item of the bracket stability.

We now make  a sequence of  modifications of $\gamma_{1}$, which will produce a sequence of partitions $(\gamma_{j})_{j \geq 1}$. All the modifications will consist in copying and painting from a source fibre to a target fibre, both meeting the same set $S^\ell(B)$ for some element $B$ of the partition $\beta_{N'_{n}}$, and some $\ell\in\{0,\dots,n-1\}$. By the choice of $N'_{n}$,
 two such fibres belong to the same column of the tower, thus this construction will abide by the general principle.
The sequence $(\gamma_{j})_{j \geq 1}$ will meet the three following properties.
For every $n$-block $A$ of $\hat \alpha_{n}$, for every $B \in \beta_{N'_{n}}$,
for every $j \geq 1$,
\begin{enumerate}
\item[(i)] if $X_{A, \gamma_{j}}$ meets $B$ then $X_{A , \gamma_{j+1}}$ fills up $B$; 
\item[(ii)] if $S^\ell(B) \subset t'_{n}$ for some $\ell \in\{1,\dots, n-1\}$ and if $A = [A_{1},A_{2} ]_{\ell}$ for some $n$-blocks $A_{1},A_{2}$ of $\hat \alpha_{n}$ such that $X_{A_{1}, \gamma_{j}}$ and $X_{A_{2}, \gamma_{j}}$ meet $B$, then $X_{A , \gamma_{j+1}}$ meets $B$;
\item[(iii)] in the complement of $T(Z_{2})$, the partitions $\gamma_{j}$ and $\gamma_{j+1}$ coincide.
\end{enumerate}

We now explain the construction of the partition $\gamma_{2}$ as a modification of $\gamma_{1}$. We first select a subset $Z'_{2}\subset Z_2$ such that both $Z'_2$ and $Z_2\setminus Z'_2$ fill up $Y$.
We further partition $Z'_{2}$ into sets $Z'_{2}(A)$ indexed by $n$-blocks of $\hat \alpha_{n}$
so that each set $Z'_{2}(A)$ still fills up $Y$. For each $n$-block $A$ and each $B \in \beta_{N'_{n}}$ we will modify the partition $\gamma_{1}$ 
along the set $T(Z'_{2}(A) \cap B)$.
Note that because of the way the sets $Z'_{2}(A)$ have been chosen, the sets $T(Z'_{2}(A) \cap B)$ for different pairs $(A,B)$ are mutually disjoint,
so that all these modifications will be compatible.
We now proceed with the modification.
\begin{itemize}
\item[--] If the set $X_{A, \gamma_{1}}$ meets $B$
then we choose one point $x$ in $X_{A, \gamma_{1}} \cap B$ and, for every $0 \leq \ell \leq n-1$, we copy the $\gamma_{1}$-name of the fibre of $S^\ell(x)$ onto every fibre that meets $S^\ell(Z'_{2}(A) \cap B)$.
In this case the copying process relative to $(A,B)$ is over.

\item[--] In the opposite case we examine the following possibility.
Suppose that $S^\ell(B) \subset t'_{n}$ for some $\ell \in\{1,\dots, n-1\}$ and that
$A = [A_{1},A_{2} ]_{\ell}$ for some $n$-blocks $A_{1},A_{2}$ of $\hat \alpha_{n}$ such 
$X_{A_{1}, \gamma_{1}}$ and $X_{A_{2}, \gamma_{1}}$ meet $B$.
Then we choose two points $x_{1} \in X_{A_{1}, \gamma_{1}} \cap B$ and
$x_{2} \in X_{A_{2}, \gamma_{1}} \cap B$; for every $0 \leq k < \ell$
we copy the  $\gamma_{1}$-name of the fibre of $S^k(x_{1})$ on
every fibre that meets $S^k(Z'_{2}(A) \cap B)$; and  for every $\ell \leq k \leq n-1$
we copy   the  $\gamma_{1}$-name of the fibre of $S^k(x_{2})$ on
every fibre that meets $S^k(Z'_{2}(A) \cap B)$.
\end{itemize}

These modifications, when made for every pair $(A,B)$, produce a new partition $\gamma_{2}$
such that properties (i), (ii) and (iii) above are (clearly) satisfied.
Then we construct the partition $\gamma_{3}$ from $\gamma_{2}$ in the same way we got $\gamma_{2}$ from $\gamma_{1}$,
with the set $Z_{2}$ replaced by $Z_{2} \setminus Z'_{2}$, and so on to produce all the $\gamma_{j}$'s.

Let $\Xi_{j}$ be the set of pairs $(A,B)$ such that $X_{A, \gamma_{j}}$ meets $B$.
Thanks to property (i), the sequence $(\Xi_{j})_{j \geq 1}$ is non-decreasing.
Since the number of $n$-blocks of $\hat \alpha_{n}$ and the partition
$\beta_{N'_{n}}$ are finite, there must be an integer $k>1$ such that
$\Xi_{k} = \Xi_{k-1}$. Then we define $\alpha_{n} = \gamma_{k}$. From property~\ref{Z.2} required on $Z_2$ and from property (iii) of the construction,
we note that $\hat \alpha_n$ and $\alpha_n$ coincide outside a union of fibres of $t'_n$
having measure less than $\eta$.

\paragraph{II.d. Fullness and bracket stability.---}
Let us check that the $n$-cylinders of $\alpha_{n}$ are $\beta_{N'_{n}}$-full.
If the cylinder $X_{A,\alpha_n}$ intersects $B\in \beta_{N'_n}$ for some $n$-block $A$ of $\alpha_n$,
then the pair $(A,B)$ belongs to $\Xi_{k}$. By definition of $k$, it also belongs to $\Xi_{k-1}$, so $X_{A,\gamma_{k-1}}$ meets $B$. Property (i) now implies that $X_{A,\alpha_n}=X_{A,\gamma_k}$ fills up $B$ as required.

The $n$-cylinders of $\alpha_n$ are also $(\beta_{N'_n},t'_n)$-stable under bracket.  Indeed the first item of the definition is satisfied by the choice of $N'_{n}$ at step II.c. For the second one, suppose that $S^\ell(B) \subset t'_{n}$ for some $\ell\in \{1,\dots,n-1\}$, $B\in \beta_{N'_n}$ and consider some $n$-cylinders
$A_1$, $A_2$, $A = [A_{1},A_{2} ]_{\ell}$ of $\alpha_{n}$ such that $X_{A_{1}, \alpha_{n}}$ and $X_{A_{2}, \alpha_{n}}$ meet $B$. Then, the definition of $k$ implies that  $X_{A_{1}, \gamma_{k-1}}$ and $X_{A_{2}, \gamma_{k-1}}$
also meet $B$. Property (ii) thus gives that $X_{A,\alpha_n}=X_{A,\gamma_k}$ meets~$B$.

In order to get the $\beta_{N_n}$-fullness and the $(\beta_{N_n},t_n)$-stability under bracket it will be now enough to choose $t_n\subset t'_n$ and $N_n\geq N'_n$ such that $S^\ell(\beta_{N_{n}})$ refines the partition of $Y$ associated to the tower $t_{n}$ for each $0 \leq \ell \leq n-1$ (see the monotonicity property following definition of bracket stability at section~\ref{ss.statement-open}). This will guarantee property 3 of the sub-lemma for $i=n$

\medskip

We now check property 3 for the partitions $\alpha_{i}$ with $i<n$. Let $A$ be an $i$-block of $\alpha_{i}$, and $B \in \beta_{N_{i}}$. 
\begin{description}
\item[First claim.] If  $X_{A,\alpha_{i}}$ meets $B$, then $X_{A, \hat \alpha_{i}}$ meets $B$.
\item[Second claim.]  If $X_{A, \hat \alpha_{i}}$ meets $B$, then $X_{A,\alpha_{i}}$ fills up $B$.
\end{description}
These two claims entail that the $i$-cylinders of $\alpha_{i}$ are $\beta_{N_{i}}$-full and ($\beta_{N_{i}}, t_{i}$)-stable under brackets, given that $\hat \alpha_{i}$ is.

The second claim is a direct consequence of the definition of $Z_1$ and of the fact that on $T(Z_{1})$ the partitions $\hat \alpha_{i}$ and $\alpha_{i}$ coincide. Let us prove the first claim. Remember that in each copying process the source and target fibres are in the same column of the tower $t'_{n}$ and thus have the same $\beta_{N_{i}}$-name. 
Now assume that the set $X_{A, \alpha_{i}}$ meets $B$ and let $x$ be a point in the intersection.
One may assume that $x$ does not belong to $X_{A, \hat \alpha_{i}}$
since otherwise we are done.
Thus some of the iterates $S^k(x)$, with $0\leq k<i$ have been involved in a copying process.
Two cases appear.
\begin{description}
\item[First case.] If all the iterates $S^k(x)$, with $0\leq k<i$, belong to the same fibre of the $t'_{n}$-tower, then this fibre has been given the $\hat \alpha_{i}$-name of some other fibre in the same column;
moreover this other fibre must intersect $B$
since $B \in \beta_{N_{i}}$ and fibres in the same column have the same $\beta_{N_{i}}$-name.
In this case we see that $X_{A, \hat \alpha_{i}}$ meets $B$.
\item[Second case.] If the iterates $S^k(x)$, with $0\leq k<i$,
meet  two different fibres,
then there exists $1 \leq \ell  < i$ such that the sequence $x, \dots , S^{i-1}(x)$ decomposes into two pieces:
\begin{itemize}
\item[--] $x, \dots, S^{\ell-1}(x)$ is the end of the fibre of $x$;  its $\alpha_{i}$-name coincides with the $\hat \alpha_{i}$-name of some finite orbit segment $y_{1}, \dots , S^{\ell-1}(y_{1})$;
\item[--] $S^{\ell}(x), \dots, S^{i-1}(x)$ is the beginning of the fibre of $S^{i-1}(x)$;  its $\alpha_{i}$-name coincides with the $\hat \alpha_{i}$-name of some finite orbit segment $S^{\ell}(y_{2}), \dots, S^{i-1}(y_{2})$.
\end{itemize}
The points $x,y_{1},y_{2}$ belong to the same element of the partition of $Y$ induced by the partition of the tower $t'_{n}$, and in particular $y_{1}, y_{2}$ belong to $B$ (property (B) of step I.c). 
Let $A_{1}, A_{2}$ be the  $i$-blocks such that   $y_{1} \in X_{A_{1}, \hat \alpha_{i}}$ and $y_{2} \in X_{A_{2}, \hat \alpha_{i}}$. Then 
$A = [A_{1},A_{2}]_{\ell}$. Furthermore the set $S^\ell(B)$ meets $t'_{n}$ and thus it also meets $t_{i}$.
Since  $\hat \alpha_{i}$ is ($\beta_{N_{i}}, t_{i}$)-stable under brackets 
 this entails again that $X_{A, \hat \alpha_{i}}$ meets $B$.
\end{description}
The proof of the first claim is complete, implying property 3 of the sub-lemma for
$i<n$.

\subsubsection*{Step III: construction of the set $t_n$}

\paragraph{III.a. Choice of the height $m$.---}
There are now two types of fibres of $t'_n$ whose $\alpha_n$-name is bad:
the fibres whose $\hat \alpha_n$- and $\alpha_n$-names coincide and are bad,
and some other fibres which are contained in the set $T(Z_1\cup Z_2)$.
By construction the union $\Delta$ of the bad fibres with respect to $\alpha_n$
is thus contained in a union $\Delta_1\cup \Delta_2$ where
\begin{itemize}
\item[--] $\Delta_1$ is a union of column, is uniform and has measure less than $\eta$
(by step I.a and the choice of $Z_1\cup Z_2$),
\item[--] $\Delta_2=T(Z_1\cup Z_2)$.
\end{itemize}
One chooses the integer $m$ much larger than $n$ and $L$
and such that within almost any finite segment of orbit of length greater than $m$
the proportion of points belonging to $\Delta_1$ is less than $2\eta$.

As before, a fibre of $t_n'$ contained in $Y\setminus \Delta$ has $\eta$-uniform distribution of the $n$-blocks of $\alpha_n$.

\paragraph{III.b. Choice of the set $t_{n}$.---}
We choose the set $t_{n}$ as  in the proof of sub-lemma~\ref{lem.weiss}. Furthermore we choose an integer $N_{n} \geq N'_{n}$ big enough so that $\beta_{N_{n}}$ and $t_{n}$ satisfy the first item in the definition of bracket stability (and thus the second item also, see step II.d).

\paragraph{III.c. Uniformity of the partition $\alpha_{n}$.---}
It remains to check that  $\alpha_{n}$ is $(\varepsilon_{n},  n-\mbox{blocks})$-represented by the set $t_{n}$.
The argument is roughly the same as in sub-lemma~\ref{lem.weiss}. We now categorise the sub-orbit of length $n$ in the fibre $f$ into four types (instead of three):
\begin{itemize}
\item[$(a)$] those that are included in a fibre of $t'_{n}$ contained in $Y\setminus \Delta$,
\item[$(b_1)$] those that are included in a fibre of $t'_{n}$ contained in $\Delta_1$,
\item[$(b_2)$] those that  are included in a fibre of $t'_{n}$ contained in $\Delta_2$,
\item[$(c)$] those that are not included in a fibre of the $t'_{n}$-tower.
\end{itemize}
The number of ``uncontrolled" sub-orbits of length $n$ in $f$ is 
$$
n_{0} = n_{b_1}+n_{b_2}+n_c \leq 2\eta N+  \left(\frac{N}{L} +1\right)2\max(\tau_{t'_n})+ n\frac{N}{m'} < N \frac{4\varepsilon_n}{100}
$$
(The estimates on $n_{b_1}$ and $n_c$ are the same as in sub-lemma~\ref{lem.weiss}. The estimate on $n_{b_2}$ follows from the fact that $Z_1\cup Z_2$ is disjoint from its $L$ first iterates, and that for any point $z \in Z_{1} \cup Z_{2}$, the set $T(\{z\})$ is a segment of orbit of length at most $2\max(\tau_{t'_n})+n$. The estimate of $\frac{N}{L} +1$ follows from the fact that the maximal return-time $\max(\tau_{t'_n})$ into $t'_n$
is much smaller than $\varepsilon L$ and $N$ is much larger than $L$). One concludes as before in sub-lemma~\ref{lem.weiss}. This completes the proof of sub-lemma~\ref{lem.open}.

\subsection{Construction of the uniform algebra (proof of lemma~\ref{lem.uniform-algebra-open})}
\label{ss.algebra-open}

The proof follows the lines of the proof of lemma~\ref{lem.uniform-algebra}:
we will construct a nested sequence of partitions $(\alpha_i^\infty)_{i\geq 1}$
which satisfies properties 1, 2, 3 of section~\ref{ss.algebra}, as well as
\begin{enumerate}
\item[4.] For every $i\geq 1$, the $i$-cylinders of the partition $\alpha_i^\infty$
are full.
\end{enumerate}
These partitions are obtained as limits of a triangular array of partitions
as in  section~\ref{ss.algebra}.
Contrarily to what was done in section~\ref{ss.algebra}, this time the summable sequence $(\varepsilon'_{i})_{i\geq 1}$ is not determined at the beginning of the construction. Instead the sequence is constructed inductively, as well as a sequence of positive numbers $(R_{i})_{i\geq 1}$.
The triangular array of partitions and the two sequences $(\varepsilon'_{i})_{i\geq 1}, (R_{i})_{i\geq 1}$
are required to satisfy properties (i) to (iv) of section~\ref{ss.algebra} as well as
\begin{enumerate}
\item[v.] There exists an increasing sequence of integers $(N_i)_{i\geq 1}$
such that for every $j\geq i$, the $i$-cylinders of the partition $\alpha_i^j$ are $\beta_{N_i}$-full
and $(\beta_{N_i},t_i)$-stable under bracket;
\item[vi.] $\varepsilon'_{j} + R_{j} < R_{j-1}$.
\end{enumerate}
During the induction we require some further condition on $R_{n}$ which will ensure property~4.

The meaning of $R_{n}$ is revealed by the following observation.
 After the construction has been completed, condition (vi) will give, for every $n$, $\sum_{j > n} \varepsilon'_{j} < R_{n}$. Then  condition (ii) will imply that for every $j > n \geq i$,  
 $$
\quad\quad\quad\quad\quad\quad\quad\quad\quad\quad\quad\quad\quad\quad
d_{\mathrm{part}}(\alpha_{i}^n, \alpha_{i}^j) < R_{n}.
\quad\quad\quad\quad\quad\quad\quad\quad\quad\quad\quad\quad\quad\quad (*)
 $$

As before the construction is by induction. Assume $(\alpha_{i}^j)_{ 1 \leq i \leq j \leq n-1}$,
$(\varepsilon'_{j})_{1 \leq j \leq n-1}$, $(R_{j})_{1 \leq j \leq n-1}$, $(N_{j})_{1 \leq j \leq n-1}$
have been constructed. 
First choose any positive $\varepsilon'_{n}$ less than $R_{n-1}$.
Apply sub-lemma~\ref{lem.open} to get the integer $N_{n}$ and the next row $(\alpha_{i}^n)_{1 \leq i \leq n}$ satisfying properties (i) to (v). It remains to explain the choice of $R_{n}$.
Let $i \in \{0,\dots,n\}$, let $A$ be an $i$-block of $\alpha_{i}^n$, and let $b \in \beta_{n}$; notice that, $n$ being given, there is only a finite number of choices of such $(i,A,b)$.
Consider, inside the space of  partitions $\alpha$ that have exactly the same number of elements as $\alpha_{i}^n$,  the open set $O(n,i,A,b)$ defined by the condition
$$
\nu(X_{A,\alpha} \cap b) > 0.
$$
Then $R_{n}$ is chosen so that for every $(i,A,b)$ such that $\alpha_{i}^n$ belongs to $O(n,i,A,b)$, for every  partition $\alpha$ with $d_{\mathrm{part}}(\alpha_{i}^n,\alpha) \leq R_{n}$, the partition $\alpha$ also belongs to $O(n,i,A,b)$. We also choose $R_{n}$ small enough so that $\varepsilon'_{n}+R_{n} < R_{n-1}$ (condition (vi)). This completes the inductive construction. 

The sequence $(\alpha_{i}^\infty)_{i \geq 1}$ is defined as in section~\ref{ss.algebra} and satisfies properties 1,2,3.
Now we turn to  property 4.
Let $i\geq 1$, let $A$ be some $i$-block of $\alpha_{i}^\infty$, we have to find a set $B$ containing $X_{A,\alpha_{i}^\infty}$ modulo a null-set and filled up by $X_{A,\alpha_{i}^\infty}$.
We define  $B$ as the union of the elements of $\beta_{N_{i}}$
whose intersection with  the $i$-cylinder $X_{A,\alpha_{i}^\infty}$ has positive measure.
Let $b$ be a $\cB$-measurable set included in $B$. There exists some  $n$ such that 
\begin{enumerate}
\item[(a)] $\nu(X_{A,\alpha_{i}^n} \cap B)>0$,
\item[(b)] $b$ is a union of elements of $\beta_{n}$.
\end{enumerate}
Since by property (v) the $i$-cylinders of $\alpha_{i}^n$ are $\beta_{N_{i}}$-full, condition (a) implies $\nu(X_{A,\alpha_{i}^n} \cap b)>0$. Then, using conditions (b) and $(*)$ and the choice of $R_{n}$, we get that  the partition $\alpha_{i}^\infty$ belongs to the set $O(n,i,A,b)$; in other words we also have
$\nu(X_{A,\alpha_{i}^\infty} \cap b)>0$. 
Hence $X_{A,\alpha_{i}^\infty}$ fills up $B$. This proves property~4.

As in section~\ref{ss.algebra} we now consider the algebra $\cA_0$ generated by the partitions $\alpha_i^\infty$
and the map $S$. It satisfies (as in section~\ref{ss.algebra}) all the conclusions of lemma~\ref{lem.uniform-algebra}.
Any element $X\in \cA_0$ is a finite union of iterates of $i$-cylinders of the
partition $\alpha_i^\infty$ for some $i\geq 1$.
By property 4, each of these cylinders is full.
Since fullness is preserved under iterations and finite unions,
one deduces that $X$ is full.
This concludes the proof of lemma~\ref{lem.uniform-algebra-open}.

\subsection{Proof of the relative open Jewett-Krieger theorem~\ref{theo.open}}\label{ss.jewett-open}
Following the strategy explained in section~\ref{ss.proof-Weiss},
we assume the hypotheses of the theorem, and apply lemma~\ref{lem.uniform-algebra-open} to get the algebra $\cA_0$ which is full. 
Then we construct,  as in section~\ref{ss.proof-Weiss}, $Y_{0}, \cA_{1}, \cC, \nu', g, \Pi$.
Now we have to prove that the map $\Pi:\cC\to\cK$ is open.
Since the elementary cylinders  form a basis of the topology of $\cC$, it is enough to prove that their images under $\Pi$ are open subsets of the Cantor set~$\cK$. 
Consider an elementary cylinder $A'=\mbox{Cl}(\Phi(A)) = p_{A}^{-1}(1) \cap \cC$ with $A\in\cA_{1}^*$.
Recall that by definition of $\cA_{1}$, there exists a set $A_{0} \in \cA_{0}$ such that $A = A_{0} \cap Y_{0}$.
The set $A_{0}$ is full: there exists a $\cB$-measurable set $B$ containing $A_{0}$ modulo a null-set and such that $A_{0}$ fills up $B$.
By construction of $Y_{0}$ the null-set $A_{0} \setminus B$ is disjoint from $Y_{0}$, hence $A \subset B$; note that the set $A$ still fills up $B$.
 By definition of the algebra $\cB$ the set $p(B)$ is clopen in $\cK$ (see section~\ref{ss.strategy}), and thus it is enough to prove that  $\Pi(A')=p(B)$.

The direct inclusion $\Pi(A')\subset p(B)$ follows immediately from the inclusion  $A\subset B$, using $\Pi\circ\Phi=p_{|Y_0}$ and the fact that $p(B)$ is closed in $\cK$.

Let us prove the reverse inclusion. The family
$$
\{p(B'), B' \in \cB^*, B' \subset B\}
$$
 form a basis of open sets in $B$. Since $A'$ is compact the set $\Pi(A')$ is closed, and thus it is enough to prove that it meets every  set of this family. Since $A$ fills up $B$ it meets every $B' \in \cB^*$ included in $B$. Consequently $\Pi(A')$ meets $p(B')$, as required.
This completes the proof of theorem~\ref{theo.open}. 

\pagebreak
\section{Bi-ergodic version}

\subsection{Statement and strategy}

\begin{theo}\label{theo.bi-ergodic}
Suppose that we are given 
\begin{itemize}
\item[--] a homeomorphism $f$  on a Cantor set $\cK$, 
which admits exactly two ergodic measures $\mu$ and $\delta _{\infty}$,
where $\mu$ has full support and $\delta _{\infty}$ is a Dirac measure at a point of $\cK$;
\item[--] an ergodic  measured dynamical system $(Y,\nu,S)$ on a standard Borel space, and a measurable map $p : Y \ra \cK$ such that
$p_{*}\nu=\mu$ and $f \circ p = p \circ S$.
\end{itemize}
Then there exist a homeomorphism $g$  on a Cantor set  $\cC$
which preserves an ergodic measure $\nu'$,
an isomorphism $\Phi$ between $(Y,\nu,S)$ and $(\cC, \nu',g)$,
an onto continuous map $\Pi\colon\cC\to\cK$ and a full measure subset $Y_0$ of $Y$
such that we have the same commutative diagram as in proposition~\ref{p.bi-ergodic-ouvert}.
Furthermore, $\Pi^{-1}(\infty)$ is a singleton set and $g$ preserves exactly two ergodic measures: the first one is $\nu'$ and
the other one is the Dirac measure on $\Pi^{-1}(\infty)$.
\end{theo}

In what follows, unless otherwise explicitly stated, all measurements are made with respect to the measure $\nu$ (and not with respect to $\delta_\infty$). We try to follow the strategy of the proof of Weiss theorem.
This time the clopen sets of $\cK$ are not uniform but only ``sub-uniform'' (see below).
The key point is that many arguments in the proof of Weiss theorem only requires sub-uniformity and not uniformity.
\begin{defi}
A set $X\subset Y$ is $(\varepsilon,N)$-\emph{sub-uniform} if for $\nu$-a.e. $y\in Y$, for every $n \geq N$, the proportion of the finite orbit $y, \dots, S^{n-1}(y)$ in $X$ is less than $\nu(X)+\varepsilon$. 
The set $X$  is \emph{sub-uniform} if for every $\varepsilon >0$ there exists $N> 0$ such that $X$ is $(\varepsilon,N)$-sub-uniform.
\end{defi}

Under the hypotheses of the theorem, note that any  clopen set $K$ of $\cK$ that does not contain the point $\infty$ is a sub-uniform set (with respect to the map $f$ and the measure $\mu$). Indeed, otherwise we would find, as a limit of Birkhoff sums of Dirac measures, an $f$-invariant  measure $\mu'$ such that $\mu'(K)$ is strictly greater than both $\mu(K)$ and $\delta_{\infty}(K)=0$, implying the existence of a third ergodic measure, a contradiction.

Consider as before  the algebra  $\cB = \{p^{-1}(K), K \subset \cK \mbox{ is a clopen set\}}$.
It is equipped with a \emph{Dirac measure} $\delta$
(\emph{i.e.} an additive function on $\cB$ with values in $\{0,1\}$)
defined by $\delta(p^{-1}(K))=1$ if and only if the clopen set $K$ contains $\infty$. Then every element $B$ of $\cB$ such that $\delta(B)=0$ is  sub-uniform (with respect to the map $S$ and the measure $\nu$).
This motivates the following lemma.
\begin{lemma}
\label{lem.sub-uniform}
Let $(Y,\nu, S)$ be an ergodic measured dynamical system on a standard Borel space.
Suppose that $\cB$ is an algebra on $Y$ equipped with a Dirac measure $\delta$ such that
\begin{enumerate}
\item $\cB$ is a countable family of measurable sets; it has no atom, and every element of $\cB$ with $\delta$-measure $1$ has positive $\nu$-measure;
\item $\cB$ is $S$-invariant and its elements with $\delta$-measure $0$ are sub-uniform sets.
\end{enumerate}
Then there exists an algebra $\cA_0$ which contains $\cB$, which is equipped with a Dirac measure that extends $\delta$, which still satisfies properties 1 and 2 above, and which in addition
generates the $\sigma$-algebra of $Y$ modulo a null-set.
Furthermore for every element $X\in \cA_0$ either $X$ or $Y \setminus X$ is contained in a $\cB$-measurable set with $\delta$-measure $0$ (this last property characterises the extension of $\delta$ to $\cA_0$).
\end{lemma}

From now on we consider an algebra $\cB$ satisfying the hypotheses of the lemma. As in section~\ref{s.weiss} we will furthermore assume that $\cB$ has no non-empty null-set  (see the remark following lemma~\ref{lem.uniform-algebra}).
In particular for every non-empty $\cB$-measurable set $t$, the return-time function of $t$ is bounded from above.
For the next definition, remember that we can see a partition with $n$ elements as a map from $Y$ to the set $\{1,\dots,n\}$. 
\begin{defi}
A partition $\alpha$ is called a \emph{marked partition} if $a=\alpha^{-1}(1)$ is $\cB$-measurable and $\delta (a)=1$.
\end{defi}

When we want to emphasise the properties of the \emph{marked element} $a$ of $\alpha$ we will also say that the couple $(\alpha,a)$ is a marked partition. Note that if $(\alpha_{i},a_{i})$ is a nested sequence of marked partitions, then  the sequence $(a_{i})_{i\geq 1}$ is decreasing.

The following  is similar to the definition of uniform distribution of $i$-blocks stated before sub-lemma~\ref{lem.weiss}; note however  that here the set $t$ has to be included in the marked element $a$, and that the tower with basis $t$ may have some column of height $1$.

\begin{defi}
Let $(\alpha,a)$ be a marked partition. A finite orbit $y, \dots , S^{N-1}(y)$ has \emph{$\varepsilon$-sub-uniform distribution of $i$-blocks of $\alpha$} if for every $i$-block $A$ of $\alpha$ that is different from $(a,\dots,a)$, 
the proportion of occurrence of this $i$-block among the $(N-i+1)$   $i$-blocks appearing in the $\alpha$-name of the orbit is less than $\nu(X_{A,\alpha}) + \varepsilon$.

A marked partition $(\alpha,a)$ is \emph{$(\varepsilon,i\mbox{-blocks})$-sub-represented by a set $t$}  if $t$ is 
included in the marked element $a$, and if 
\begin{enumerate}
\item there exists an integer $N$ such that  the return-time function $\tau$ of $t$  satisfies, for all $y \in t$, 
$$
\tau(y)  \in \{1\} \cup \left[\frac{i}{\varepsilon} , N \right],
$$ 
\item for all $y \in t$ with $\tau(y)>1$, the finite orbit $y, \dots , S^{\tau(y)-1}(y)$ has $\varepsilon$-sub-uniform distribution of $i$-blocks of $\alpha$.
\end{enumerate}
\end{defi}

For the next sub-lemma, we assume the hypotheses of lemma~\ref{lem.sub-uniform}.
We fix a sequence $(\beta_{i})_{i\geq 1}$ of nested marked partitions by $\cB$-measurable sets.

\begin{sublemma}\label{lem.weiss-bi-ergodic}
For $n\geq 1$, let $(\hat \alpha_{i})_{1\leq i \leq n}$ be a sequence of nested marked partitions, $(t_{i})_{0 \leq i \leq n-1}$ a decreasing sequence of $\cB$-measurable sets, and $(\varepsilon_{i})_{ 1 \leq i \leq n-1}$ a sequence of positive numbers.
Assume that $\hat \alpha_{n}$ refines $\beta_{n}$ and that for $1 \leq i \leq  n-1$,
\begin{itemize}
\item[--] the partition $\hat \alpha_{i}$ refines the partition $\beta_{i}$,
\item[--] the partition   $\hat \alpha_{i}$ is $(\varepsilon_{i}, i\mbox{-blocks})$-sub-represented by $t_{i}$.
\end{itemize}
Let $\varepsilon_{n} >0$.
Then there exist a $\cB$-measurable set $t_{n} \subset t_{n-1}$ and a sequence of nested marked partitions
$(\alpha_{i})_{1\leq i \leq n}$ such that for every $1 \leq i \leq n$,
\begin{enumerate}
\item the partition $\alpha_{i}$ refines the partition  $\beta_{i}$,
\item the partition $\alpha_{i}$ is $(\varepsilon_{i}, i\mbox{-blocks})$-sub-represented by the set $t_{i}$,
\item the partition $\alpha_{i}$ has the same number of elements as $\hat \alpha_{i}$, the same marked element,
and  $d_{\mathrm{part}}(\alpha_{i},\hat \alpha_{i}) < \varepsilon_{n}$.
\end{enumerate}
\end{sublemma}

\subsection{Proof of sub-lemma~\ref{lem.weiss-bi-ergodic}}
A \emph{marked tower} is a tower whose basis $t$ is $\cB$-measurable and satisfies $\delta(t) =1$, and whose partition has exactly one element (called the \emph{marked element}) whose  column has height $1$.
The minimal height of a marked tower is defined as the minimum of the height of the columns of height $>1$.

\begin{fact}[marked Rokhlin lemma]
For every $N>0$, and every $\cB$-measurable set $t$ with $\delta$-measure $1$, there exists a $\cB$-measurable set $t' \subset t$  with $\delta$-measure $1$ and arbitrarily small $\nu$-measure, with the following property.
Let $\tau$ denotes the return-time function on $t'$. Then there exists $M>0$ such that, on $t'$, $\tau$ takes its values in   $\{1\} \cup [N,  M]$.
\end{fact}

\begin{proof}[Proof of the fact]
Since $\cB$ has no atom, without loss of generality we can assume that $\nu(t)$ is arbitrarily small.
Consider the $\cB$-measurable set $t' := t \cap S^{-1}t \cap \cdots \cap S^{-(N-1)}t$. Then the return-time function of $t'$ takes values in $\{1, N, N+1, \dots\}$.
Since the measure $\delta$ is invariant under $S$ we have  $\delta (t') = 1$, thus    the complementary set of $t'$ is sub-uniform, which entails the boundedness of the return-time function.
\end{proof}

We now assume the hypotheses of sub-lemma~\ref{lem.weiss-bi-ergodic}. We follow the steps of the proof of sub-lemma~\ref{lem.weiss}, indicating the changes.

\subsubsection*{Step I: choice of a first tower}

\paragraph{I.a. Choice of the height.---}
We adapt the fact from the step I.a of section~\ref{ss.proof-weiss} to get the following, which is again contained in lemma 15.16 of~\cite{Glasner}.
\begin{fact}
For every partition $\alpha$, every $n>0$, and every $\eta >0$ there exists some integer $m'$ such that for every marked tower with marked element of measure less than $\frac{\eta}{2}$ and minimal height bigger than $m'$, the union of the good fibres, that is, fibres of height $>1$ having $\eta$-uniform distribution of $n$-blocks of $\alpha$, has measure greater than $1-\eta$.
\end{fact}

As before we choose a small $\eta>0$ and apply the fact to the partition $\hat \alpha_{n}$, getting a big integer $m'$.

\paragraph{I.b. Choice of the basis.---}
We apply the marked Rokhlin lemma to get a $\cB$-measurable set $t'_{n}$ included not only in $t_{n-1}$, but also in the marked element $\hat a_{n}$ of the partition $\hat \alpha_{n}$.

\paragraph{I.c. Choice of the partition.---}
We consider a tower with basis $t'_{n}$, and we choose a $\cB$-measurable partition of the basis which induces a partition $\theta$ of $Y$ satisfying  the same condition~A as before ($\theta$ refines $\beta_{n}$), as well as the following condition~:
\begin{itemize}
\item[(C)] The marked elements $\hat a_{i}$ of $\hat \alpha_{i}$, $i=1, \dots,n$, are union of elements of $\theta$.
\end{itemize}
It is possible to get property (C) since the marked elements of the partitions  $\hat \alpha_{i}$, $i=1, \dots,n$, are $\cB$-measurable sets. Conditions (A) and (C) amounts to saying that $\theta$ refines some $\cB$-measurable partition of $Y$, thus they are achieved by choosing a sufficiently fine partition of the basis.

\subsubsection*{Step II: construction of the partitions $\alpha_{1},\dots,\alpha_n$}

\paragraph{II.a.  General principle of the construction.---}
As before.

\paragraph{II.b. Easy verifications.---}
Most properties  will follow from the above general principle, with the same argument as before. 
Let us add that the partitions $\alpha_{i}$ and $\hat \alpha_{i}$ will automatically have the same marked element, due to the choice of the partition  on the tower (property C). After the explicit construction, it will remain to check the sub-uniformity of $\alpha_{n}$ (property 2 for $i=n$).

\paragraph{II.c. Construction.---}
As before.

\subsubsection*{Step III: construction of the set $t_n$}

\paragraph{III.a. Choice of the height $m$.---}
We consider the union $\Delta$ of fibres of $t'_n$ whose length is greater than $1$ and which
are bad for the partition $\alpha_n$ (i.e. which do not have $\eta$-uniform distribution of $n$-blocks of the partition $\alpha_n$). This set has $\nu$-measure less than $\eta$;
moreover it is a $\cB$-measurable set with $\delta$-measure $0$, hence it is a sub-uniform.
So there exists an integer $m$ much larger than $n$
such that within almost every finite segment of orbit of length greater than $m$
the proportion of points belonging to $\Delta$ is less than $2\eta$.

\paragraph{III.b. Choice of the set $t_{n}$.---}
As before, we apply the marked Rokhlin lemma  to get a $\cB$-measurable set $t_{n} \subset t'_{n}$ whose return-time function takes its values in $\{1\} \cup [m, M]$ for some integer $M$.

\paragraph{III.c. Sub-uniformity of the partition $\alpha_{n}$.---}
It remains to check that $\alpha_{n}$ is $(\varepsilon_{n},  n\mbox{-blocks})$-sub-represented by the set $t_{n}$.
Let us choose some specific $n$-block $A \in \alpha^n$
which  is not the $n$-block $(a,\dots,a)$, where $a$ is the marked element of $\alpha_{n}$.
We now categorise the sub-orbits of length $n$ in a fibre $f$ of $t_n$
of length greater than $1$ into four types:
\begin{itemize}
\item[$(a)$] those that are included in a fibre of length $>1$ of $t'_{n}$ contained in $Y\setminus \Delta$,
\item[$(b)$] those that are included in a fibre of length $>1$ of $t'_{n}$ contained in $\Delta$,
\item[$(c)$] those that meet a fibre of $t'_{n}$ of length $>1$ without being included in it, 
\item[$(d)$] those that are included in the union of the fibres of length $1$
(\emph{i.e.} in the marked element $t'_{n} \cap S^{-1}(t'_{n})$ of the tower).
\end{itemize}
Since the marked element of the tower $t'_n$ is included in $a$ (see step I.b),
the sub-orbits of type $(d)$ have their $\alpha_{n}$-name equal to $(a,\dots,a)$,
thus there is no occurrence of $A$ in the sub-orbits of type $(d)$.
Hence the uncontrolled orbit segments in $f$ are those of type $(b)$ or $(c)$.
The number $n_{0}=n_b+n_c$ of sub-orbits of type $(b)$ or $(c)$ satisfies the same inequality as in the proof of sub-lemma~\ref{lem.weiss}. Hence, the same estimates as in sub-lemma~\ref{lem.weiss} show that the total number of occurrences of the $n$-block $A$ in the fibre $f$ is less than $(N-n-1)(b + \varepsilon_{n})$ where $b$ is the expected proportion (as before). This shows that $\alpha_{n}$ is  $(\varepsilon_{n}, n\mbox{-blocks})$-sub-represented by the set $t_{n}$, and completes the proof of the sub-lemma.

\subsection{Construction of the sub-uniform algebra (proof of lemma~\ref{lem.sub-uniform})}
\label{ss.algebra-bi-ergodic}

The proof of lemma~\ref{lem.sub-uniform} mainly consists in copying the proof of lemma~\ref{lem.uniform-algebra} above, replacing partitions by marked partitions and uniformity by sub-uniformity. We only indicate the most substantial changes.

Given a nested sequence of $\cB$-measurable marked partitions $(\beta_i)_{i\geq 1}$ that generates $\cB$,
we have to find a nested sequence of \emph{marked} partitions $(\alpha_{i}^\infty,a_i)$
which  satisfies  properties 1 and 2 (as in section~\ref{ss.algebra}), and the following strong sub-uniformity property.
\begin{enumerate}
\item[3'.] For every $i\geq 1$, and for every $i$-block $A\neq (a_i,\dots,a_i)$ of $\alpha_i^\infty$,
the corresponding $i$-cylinder $X_{A,\alpha_i^\infty}$ is sub-uniform.
\end{enumerate}
We start the construction by choosing a sequence of nested marked partitions $(\gamma_i)_{i\geq 1}$
that generates the $\sigma$-algebra of $Y$ and such that $\gamma_i$ refines $\beta_i$ for each $i\geq 1$.
The key property for getting such a sequence is that the intersections of the marked elements of the partitions $\beta_{n}$ has $\nu$-measure $0$ (since $\cB$ has no atom).
Then we produce an array of partitions $(\alpha_i^j)_{1\leq i\leq j}$
satisfying properties (i) to (iv) as before (replacing ``represented" by ``sub-represented" in
property (iv)).
We add that all the partitions in the same column have the same marked element $a_i$,
as will be granted by point 3 of sub-lemma~\ref{lem.weiss-bi-ergodic}.
This will clearly imply that the limit partition $\alpha_{i}^\infty$ is a marked partition
with marked element $a_i$.
The properties 1, 2 and 3' are now checked as in section~\ref{ss.algebra}.

As before $\cA_{0}$ is the algebra generated by the partitions $\alpha_{i}^\infty$ and the dynamics.
It is countable, $S$-invariant, contains $\cB$, has no atom and no null-set, generates the
$\sigma$-algebra of $Y$ modulo a null-set.
It remains to prove that any $X\in \cA_0$ has the following property.
\begin{itemize}
\item[(*)] Either $X$ or $Y\setminus X$ is contained in a $\cB$-measurable set with $\delta$-measure zero.
Furthermore, in the first case, $X$ is sub-uniform.
\end{itemize}
For this we first consider the case when $X$ is an $i$-cylinders of $\alpha_{i}^\infty$ for some $i$. Two situations may occur.
\begin{itemize}
\item[--] Either $X$ is the ``marked cylinder'' $X_{A,\alpha_i^\infty}$
corresponding to the $i$-block $A=(a_i,\dots,a_i)$ where $a_i$ is the marked element of $\alpha_{i}^\infty$.
Then $Y\setminus X$ belongs to the algebra $\cB$ and has $\delta$-measure zero.
\item[--] Or $X$ is contained in the complementary set $Y\setminus X_{A,\alpha_i^\infty}$
of the marked cylinder, which is a $\cB$-measurable set of $\delta$-measure $0$.
Furthermore $X$ is a union of cylinders and is sub-uniform, thanks to the ``strong sub-uniformity property''.
\end{itemize}
In both cases $X$ satisfies  property (*). Now this property is invariant under $S$
and preserved by finite disjoint unions, and thus it is satisfied by every element $X\in\cA_{0}$.

\subsection{Proof of the bi-ergodic relative Jewett-Krieger theorem~\ref{theo.bi-ergodic}}\label{ss.jewett-bi-ergodic}

Again the proof of theorem~\ref{theo.bi-ergodic} follows the line of the proof of theorem~\ref{theo.weiss}.
More precisely, the definitions of $\cC$, $\nu'$, $g$ and $\Pi$ are exactly the same.
These definitions being given, it remains to prove that $\Pi^{-1}(\infty)$ is a single point $\infty'\in\cC$, and that $\nu'$ is ergodic and is
the unique $g$-invariant measure that projects down onto $\mu$. 

For the first point, note that if $\Pi(x)=\infty$ for some $x\in \cC$, then $x$ belongs to all the elementary cylinders $\Phi(B)$ with $B\in \cB$ such that $\delta(B)=1$. If one assumes by contradiction that $\Pi^{-1}(\infty)$ contains two distinct points $x,x'$, there exists an elementary cylinder $\Phi(A)=p_A^{-1}(1)\cap\cC$ which separates $x$ and $x'$, say e.g. $x\in\Phi(A)$ but $x'\notin\Phi(A)$. Since $\Pi(x)=\infty$, we have $\delta(A)=1$. By lemma~\ref{lem.sub-uniform}, there exists $B\subset A$ in $\cB$ such that $\delta(B)=1$. The point $x'$ does not belong to $A$, and thus, does not belong to $B$. This gives the desired contradiction.

Let us check the second point. Using an argument entirely similar to the proof of unique ergodicity in theorem~\ref{theo.weiss} (substituting as usual uniformity by sub-uniformity), we first note that for any elementary cylinder $A' \subset \cC$ that does not contain the point $\infty'$, for any point $c\in \cC$,
the upper limit of the sequence of Birkhoff averages $(\frac 1 n S_{n}(1_{A'},c))_{n \geq 1}$ is less than $\nu'(\cC')$.
Now let us consider an ergodic measure $\nu''$  for $g$ such that $\nu''(\infty')=0$, and let us prove that $\nu''=\nu'$. For this it suffices to see that $\nu''(A') = \nu'(A')$ for any elementary cylinder $A'$ which does not contains $\infty'$.
Using the Birkhoff sums of a point $c$ that is typical for $\nu''$, we get that $\nu''(A') \leq \nu'(A')$.
For any $\varepsilon>0$, one may write the complement of $A'$ as a disjoint union of two elementary cylinders $A'_{1}\cup A'_{2}$
such that $A'_{1}$ contains $\infty'$ but has $\nu''$ measure less than $\varepsilon$.
As for $A'$ one has $\nu''(A_{2})\leq \nu'(A_{2})$. One deduces $\nu''(\cC\setminus A')\leq \nu'(\cC\setminus A')+\varepsilon$.
Since this holds for any $\varepsilon$, one deduces $\nu''(A')\geq \nu'(A')$. We get $\nu''(A') = \nu'(A')$ as required. 

Since $\nu'$ is the unique measure that projects on the ergodic measure $\mu$ by $\Pi$, it is ergodic.
This completes the proof.

\section{Bi-ergodic open version}
\label{sec.bi-ergodic-open}

\subsection{Statement and strategy}
Proposition~\ref{p.bi-ergodic-ouvert} can now be restated as follows.

\begin{theo}\label{theo.bi-ergodic-open}
In the statement of theorem~\ref{theo.bi-ergodic} one can add that the map $\Pi$ is open.
\end{theo}

The reader 
 will not be surprised by the following statement.
 
\begin{lemma}\label{lem.sub-uniform-open}
In the statement of lemma~\ref{lem.sub-uniform}, one can add that every element of the algebra $\cA_0$ is full.  
\end{lemma} 

We state a new version of sub-lemma~\ref{lem.weiss-bi-ergodic}.
\begin{sublemma}\label{lem.weiss-bi-ergodic-open}
For $n \geq 1$, let $(\hat \alpha_{i})_{ 1 \leq i \leq n}$ be a sequence of nested marked partitions,
$(t_{i})_{0 \leq i \leq n-1}$ a decreasing sequence of $\cB$-measurable sets, 
$(\varepsilon_{i})_{ 1 \leq i \leq n-1}$ a sequence of positive numbers,
and $(N_{i})_{1 \leq i \leq n-1}$ a increasing sequence of integers.
Assume that  $\hat \alpha_{n}$ refines $\beta_{n}$ and that for $1 \leq i \leq n-1$,
\begin{itemize}
\item[--] the partition $\hat \alpha_{i}$ refines the partition $\beta_{i}$,
\item[--] the partition $\hat \alpha_{i}$ is $(\varepsilon_{i}, i\mbox{-blocks})$-sub-represented by  $t_{i}$,
\item[--] the $i$-blocks of  $\hat \alpha_{i}$ are $\beta_{N_{i}}$-full and ($\beta_{N_{i}}, t_{i}$)-stable under brackets.
\end{itemize}
Let $\varepsilon_{n} >0$.
Then there exist a sequence of marked nested partitions $(\alpha_{i})_{1 \leq i \leq n}$,
a $\cB$-measurable set $t_{n}\subset t_{n-1}$ and an integer $N_n\geq N_{n-1}$ such that, for every $1 \leq i \leq n$,
\begin{enumerate}
\item the partition $\alpha_{i}$ refines the partition $\beta_{i}$,
\item the partition $\alpha_{i}$ is $(\varepsilon_{i}, i\mbox{-blocks})$-sub-represented by  $t_{i}$,
\item the $i$-blocks of $\alpha_{i}$ are $\beta_{N_{i}}$-full and ($\beta_{N_{i}}, t_{i}$)-stable under brackets,
\item the partition $\alpha_{i}$ has the same number of elements as $\hat \alpha_{i}$, the same marked element, and $d_{\mathrm{part}}(\alpha_{i},\hat \alpha_{i}) < \varepsilon_{n}$.
\end{enumerate}
\end{sublemma}

\subsection{Proof of sub-lemma~\ref{lem.weiss-bi-ergodic-open}}
\subsubsection*{Step I: choice of a first tower}

\paragraph{I.a. Choice of the height.---}
As in the bi-ergodic case.

\paragraph{I.b. Choice of the basis.---}
As in the bi-ergodic case.

\paragraph{I.c. Choice of the partition.---}
As in the open case, with the addition that the partition $\theta$ must satisfy condition~C from the bi-ergodic case.

\subsubsection*{Step II: construction of the partitions $\alpha_{1},\dots,\alpha_n$}

\paragraph{II.a.  General principle of the construction.---}
As before.

\paragraph{II.b. Easy verifications.---}
As in the bi-ergodic case.

\paragraph{II.c. Construction.---}
As in the open case.

\paragraph{II.d. Fullness and bracket stability.---}
Almost as in the open case, with the following change in the Second case of the proof of the First claim. This change is made necessary by the existence of  some columns of height one in the tower $t'_{n}$, allowing the orbit of the point $x$ between time $0$ and $i-1$ to spend some time in the basis. We replace the Second case by the following.
\begin{description}
\item[Second case.] If the iterates $S^k(x)$, with $0\leq k<i$,
meet at least two different fibres,
then there exist $1 \leq k_{1} \leq k_{2}  < i$ such that the sequence $x, \dots , S^{i-1}(x)$ decomposes into three pieces:
\begin{itemize}
\item[--] $x, \dots, S^{k_{1}-1}(x)$ is the end of the fibre of $x$;  its $\alpha_{i}$-name coincides with the $\hat \alpha_{i}$-name of some sequence $y_{1}, \dots , S^{k_{1}-1}(y_{1})$ for some point $y_{1} \in B$;
\item[--] $S^{k_{1}}(x), \dots, S^{k_{2}-1}(x)$ (which may be empty) is contained in $t'_{n}$;
\item[--] $S^{k_{2}}(x), \dots, S^{i-1}(x)$ is the beginning of the fibre of $S^{i-1}(x)$;  its $\alpha_{i}$-name coincides with the $\hat \alpha_{i}$-name of some sequence $S^{k_{2}}(y_{2}), \dots, S^{i-1}(y_{2})$ for some point $y_{2} \in B$.
\end{itemize}
The points $x,y_{1},y_{2}$ belong to the same element of the partition of $Y$ induced by the partition of the tower $t'_{n}$. Thus the fibres of the points $S^k(x), S^k(y_{1}),S^k(y_{2})$ have the same height, in particular the points $S^{k_{1}}(y_{1}), \dots, S^{k_{2}-1}(y_{1})$ and $S^{k_{1}}(y_{2}), \dots, S^{k_{2}-1}(y_{2})$
also belong to $t'_{n}$. We recall that $t'_{n}$ is included in the marked element of $\hat \alpha_{i}$, and thus also in the marked element of $\alpha_{i}$.
Let $A_{1}, A_{2}$ be the  $i$-blocks such that   $y_{1} \in X_{A_{1}, \hat \alpha_{i}}$ and $y_{2} \in X_{A_{2}, \hat \alpha_{i}}$. Then 
$A = [A_{1},A_{2}]_{k_{1}}$.
Since $\hat \alpha_{i}$ is ($\beta_{N_{i}}, t_{i}$)-stable under brackets, this entails again that $X_{A, \hat \alpha_{i}}$ meets $B$.
\end{description}

\subsubsection*{Step III: construction of the set $t_n$}

\paragraph{III.a. Choice of the height $m$.---}
The union $\Delta$ of the fibres of the tower $t_n'$ of height $>1$ and whose $\alpha_n$-name is bad
is now contained in the union $\Delta_1\cup \Delta_2$ of two sets:
\begin{itemize}
\item[--] $\Delta_1$ is a union of columns of $t'_n$. It has measure less than $\eta$. Moreover is $\cB$-measurable, has $\delta$-measure $0$ and hence is sub-uniform.
\item[--] $\Delta_2$ is the set $T(Z_1\cup Z_2)$.
\end{itemize}
One chooses the integer $m$ as in the open case.
As before, the fibres of the tower $t_n'$ included in $Y\setminus \Delta$ have $\eta$-uniform distribution of the $n$-blocks of $\alpha_n$.

\paragraph{III.b. Choice of the set $t_{n}$.---}
As in the bi-ergodic case.

\paragraph{III.c. Sub-uniformity of the partition $\alpha_{n}$.---}
It remains to check that   $\alpha_{n}$ is $(\varepsilon_{n},  n\mbox{-blocks})$-sub-represented by the set $t_{n}$.
As in the bi-ergodic case, we choose some specific $n$-block $A \in \alpha^n$, and suppose that this is not the $n$-block $(a,\dots,a)$ where $a$ is the marked element of $\alpha_{n}$.
We categorise the sub-orbits of length $n$ included in a fibre $f$ of length $>1$ into five types:
\begin{itemize}
\item[$(a)$] those that are included in a fibre of $t'_{n}$ of length $>1$ contained in $Y\setminus \Delta$,
\item[$(b_1)$] those that are included in a fibre of $t'_{n}$ of length $>1$ contained in $\Delta_1$,
\item[$(b_2)$] those that are included in a fibre of $t'_{n}$ of length $>1$ contained in $\Delta_2$,
\item[$(c)$] those that meet a fibre of $t'_{n}$ of length $>1$ without being included in it, 
\item[$(d)$] those that are included in the fibres of length $1$ (i.e. in the marked element $t'_{n} \cap S^{-1}(t'_{n})$ of the tower).
\end{itemize}
As in the bi-ergodic case, there is no occurrence of $A$ in the orbit segments of type $(d)$.
Hence the uncontrolled sub-orbits of length $n$ have type $(b_1)$, $(b_2)$ or $(c)$. The number $n_{0}=n_{b_1}+n_{b_2}+n_c$ of sub-orbits of type $(b_1)$, $(b_2)$ or $(c)$ satisfies the same inequality as in the open case. Hence, the same estimates as before show that the total number of occurrences of the $n$-block $A$ in the fibre $f$ is less than $(N-n-1)(b + \varepsilon_{n})$ where $b$ is the expected proportion. This shows that $\alpha_{n}$ is  $(\varepsilon_{n}, n\mbox{-blocks})$-sub-represented by the set $t_{n}$, and completes the proof of the sub-lemma.

\subsection{Construction of the sub-uniform algebra (proof of lemma~\ref{lem.sub-uniform-open})}
One has to modify the proof of lemma~\ref{lem.sub-uniform} as one modified
the proof of lemma~\ref{lem.uniform-algebra} at section~\ref{ss.algebra-open}:
one build a nested sequence of marked partitions $(\alpha_i^\infty)$
that satisfies properties 1, 2, 3', 4 and 5 from sections~\ref{ss.algebra}, \ref{ss.algebra-open}
and~\ref{ss.algebra-bi-ergodic}.
The algebra $\cA_0$ generated by these partitions and $S$ then satisfies all the required properties.

\subsection{Proof of the bi-ergodic open relative Jewett-Krieger theorem~\ref{theo.bi-ergodic-open}}
One repeats the proof of section~\ref{ss.jewett-bi-ergodic}.
In order to check that the map $\Pi$ is open using that the algebra $\cA_{0}$ is full,
the argument is the same as in section~\ref{ss.jewett-open} (the dynamics is not used here). This completes the proof of theorem~\ref{theo.bi-ergodic-open} (and thus of proposition~\ref{p.bi-ergodic-ouvert}).

\part{Semi-conjugacies: from open  to fibred}\label{B}

In this part we complete the proof of theorem~\ref{theo.bi-ergodic-fibred}: we explain how to turn the open  semi-conjugacy provided by proposition~\ref{p.bi-ergodic-ouvert}
 into a skew-product. Contrarily to what precedes, here the results and methods are purely topological.
 
\begin{propalpha}\label{prop.open-fibred}
Let $g:\cC \ra \cC$ and $f:\cK \ra \cK$ be two homeomorphisms on two Cantor sets, and $\Pi : \cC \ra \cK$ a continuous surjective map between such that $f\Pi = \Pi g $. Suppose in addition that $\Pi$ is an open map.

Then there exist a Cantor set $\cQ$, a homeomorphism $g_1 : \cK \times \cQ \ra \cK \times \cQ$ and a continuous injective map $\Phi : \cC \ra \cK \times \cQ$  such that the diagram below commutes (where $\pi_{1}: \cK \times \cQ \to\cK$ is the projection on the first coordinate). Furthermore, every point of $\cK \times \cQ$ that is not in the image of $\Phi$ is a wandering point of $g_1$, In particular every $g_1$-invariant measure is supported on~$\Phi(\cC)$.
$$
\xymatrix{ 
&   \cK \times \cQ  \ar@{.>}[rr]  ^{g_1}    \ar@{.>}[ddddl] ^{\pi_{1}} & & \cK \times \cQ  \ar@{.>}[ddddl]  ^{\pi_{1}} \\ 
\cC \ar@{.>}[ur]    ^{\Phi}   \ar@{>}[rr]   ^(.65){g}    \ar@{>}[ddd]   ^(.4){\Pi}  & & \cC \ar@{.>}[ur]    ^{\Phi}   \ar@{>}[ddd]   ^(.4){\Pi} \\
\\
\\
\cK \ar@{>}[rr]   ^{f}    & & \cK 
}$$
\end{propalpha}

We will deduce proposition~\ref{prop.open-fibred} from the following lemma.
\begin{lemma*}
Let $f$ be a homeomorphism of a Cantor set $\cK$. Let $\cQ$ be another Cantor set, and $g_0$ be a homeomorphism of the product $\cK\times \cQ$ which fibres over $f$. Let $F$ be a closed $g_{0}$-invariant subset of $\cK\times \cQ$. Assume that, for every $\theta\in\cK$, the closed set $F_\theta:=F\cap (\{\theta\}\times \cQ)$ has empty interior in $\{\theta\}\times \cQ$. Also assume that the first projection $\pi_{1} : \cK\times \cQ \to \cK$ restricts to an open map $\pi_{1} : F \to \cK$.
Then there exists a homeomorphism $g_{1}:\cK\times \cQ\to \cK\times \cQ$ with the following properties:
\begin{itemize}
\item[--] $g_{1}$ fibres over $f$;
\item[--] $g_{1}=g_0$ on $F$;
\item[--] the non-wandering set of $g_{1}$ is included in $F$.
\end{itemize}
\end{lemma*}

\subsection*{Proof of proposition~\ref{prop.open-fibred} assuming the lemma}

Let $f:\cC\to\cC$, $g:\cK\to\cK$ and $\Pi:\cC\to\cK$ be as in the statement of the proposition. 
We choose an auxiliary Cantor set $\cD$, and we pick a point $d_0$ in $\cD$. Then we consider the homeomorphism $g_{0} = f \times g\times \mbox{Id}$ on the Cantor set $\cK \times \cC\times \cD$. The map $\Phi: \cC \to  \cK \times \cC \times \cD$ given by $\Phi(c) = (\Pi(c),c,d_0)$ is continuous and injective. The set
$$F :=\Phi(\cC) = \{(\theta,x,d_0)\in\cK\times\cC\times\cD \mbox{ such that } \Pi(x)=\theta \}$$
is invariant under $g_{0}$. 
Thus if we denote the Cantor set $\cC \times \cD$ by $\cQ$, we see that the map $g_{0}$ satisfies all properties required by proposition~\ref{prop.open-fibred} apart from the fact that points outside the image of $\Phi$ are wandering. This last property will be achieved by modifying $g_{0}$ outside the set $\Phi(\cC)$, thanks to the lemma stated above. Observe that  the map $g_0$ does satisfy the hypothesis of this lemma: indeed, the intersection of $F$ with any fibre $\{\theta\}\times \cQ$ of the first projection $\pi_{1}:\cK \times \cQ \to \cK$ is contained in $\{\theta\}\times \cC\times \{0\}$ and has empty interior in  the fibre (this was the reason for considering the product $\cC\times\cD$ instead of $\cC$). Therefore, proposition~\ref{prop.open-fibred} will follow from the lemma.

\subsection*{Proof of the lemma} 

The proof of the lemma
 occupies the remainder of part~\ref{B}. During this proof, we will construct many partitions of the Cantor set $\cK\times\cQ$. Unless otherwise explicitly stated, the elements of these partitions will be products of clopen (closed and open) subsets of $\cK$ by clopen subsets of~$\cQ$.
 
\paragraph{Construction  of a sequence of partitions $(\cP_n)_{n\geq 0}$ of $\cK\times\cQ$.}~
We will construct a sequence  $(\cP_{n})_{n\geq 0}$ of finite partitions of $\cK\times \cQ$ satisfying the following properties.
\begin{enumerate}
\item For every $n\in\NN$, the elements of the partition $\cP_n$ are products of clopen subsets of $\cK$ by clopen subsets of $\cQ$.
\item For every $n\in\NN$, the partition $\cP_{n+1}$ refines the partition $\cP_n$, and the family of all the elements of all the partitions $\cP_{n}$, $n\geq 0$, is a basis for the topology: for every point $x$ in $\cK\times \cQ$, if we denote by $P_n(x)$ the element of the partition $\cP_n$ containing $x$, then the decreasing sequence $(P_{n}(x))_{n\geq 0}$ is a basis of neighbourhoods of $x$.
\item~\label{pn.open} For every $n$ and every set $P \in \cP_{n}$, if $P$ meets $F$ then the projection $\pi_{1}(F \cap P)$ coincides with $\pi_{1}(P)$.
\item~\label{pn.last} For every $n$ and every set $P\in \cP_{n}$, for every $\theta\in\pi_1(P)\subset\cK$, there exists an element $\check P \in\cP_{n+1}$, included in $P$, such that:
\begin{itemize}
\item[(i)] $\pi_{1}(\check P)$ contains $\theta$,
\item[(ii)] $\check P$ is disjoint from $F$.
\end{itemize}
\end{enumerate}

The construction of the sequence $(\cP_n)_{n\geq 0}$ has two steps. We first construct a sequence $(\cP'_{n})_{n\geq 0}$ satisfying properties 1, 2 and 4 but maybe not 3. Such a sequence is constructed by induction: we get $\cP'_{n+1}$ from $\cP'_{n}$ by   subdividing every element $P \in \cP'_{n}$  in order to guarantee  property 4, making use of the hypothesis that every set $F_{\theta}$ has empty interior in $\{\theta\} \times \cQ$.
Then the partition $\cP_{n}$  is obtained from $\cP'_{n}$ by dividing each element into two pieces, in the following way. For each element $P = a \times b \in \cP'_{n}$ (with $a\subset\cK$ and $b\subset\cQ$), we observe that the set $a_{1} = \pi_{1}(P \cap F)$ is a clopen subset of $\cK$,  thanks to the hypothesis that $\pi_{1} : F \to \cK$ is an open map. Let $a_{2}$ be the clopen set $a \setminus a_{1}$; then we divide $P$ into the two sets $P_{1}= a_{1} \times b$, $P_{2}= a_{2} \times b$, and define the partition $\cP_{n}$ to be the family of all such sets $P_{1}$ and $P_{2}$.

\paragraph{Construction  of the sequence of sets $(X_n)_{n\in\ZZ}$.}~
Now, we want to define a sequence of sets $(X_{n})_{ n \in \ZZ}$ that constitutes an infinite partition of the complement of $F$, with the idea that the homeomorphism $g_{1}$ required by the lemma will map each set $X_{n}$ onto the set $X_{n+1}$. The construction is illustrated by the left-hand part of figure~\ref{figure}.

First, for every $n\in\NN$, we define $V_n$ to be the union of the elements of the partition $\cP_n$ which have a non-empty intersection with $F$. Clearly, $V_n$ is a clopen neighbourhood of $F$.

The set $X_{0}$ is defined as $V_{0} \setminus V_{1}$.
For $n\in\NN\setminus\{0\}$, we define the sets $X_n, X_{-n}$ as follows.
First observe that the set $V_n\setminus V_{n+1}$ is a union of elements $P$ of the partition $\cP_{n+1}$. Now, consider such an element $P$.
It is the product of a clopen subset $a$ of $\cK$ by a clopen subset $b$ of $\cQ$. We divide $b$ into two non-empty clopen subsets $b^-$ and $b^+$, and we set $P^-:=a\times b^-$ and $P^+:=a\times b^+$. And we define, for $n>0$,
\begin{itemize}
\item[--] the set $X_{-n}$ to be the union of all the sets $P^-$ when $P$ ranges over the elements of the partition $\cP_{n+1}$ contained in $V_n\setminus V_{n+1}$;
\item[--] the set $X_{n}$ to be the union of all the sets $P^+$ when $P$ ranges over the elements of the partition $\cP_{n+1}$ contained in $V_n\setminus V_{n+1}$.
\end{itemize}
For every $n \in \ZZ$, the set $X_{n}$ is equipped with the partition $\cP(X_{n})$ induced by  the partition $\cP_{\mid n \mid +1}$.

\bigskip

Our goal is now to construct a homeomorphism $g_1:\cK\times C\to\cK\times C$ such that $g_1=g_0$ on $F$ and, for every $n\in\ZZ$, $g_1(X_n)=X_{n+1}.$ The next paragraph may be seen as a first approximation to $g_{1}$, selecting for each point $x$ a set of points in which, in the last paragraph,  $g_{1}(x)$ will be chosen.

\paragraph{Pairing between $X_{n}$ and $X_{n+1}$}~
We fix some integer $n \in \ZZ$.
We now define a partition $\cX_{n}$ on $X_{n}$ and  a partition $\cX'_{n}$ on $X_{n+1}$ that are adapted to the map $f : \cK \to \cK$. 
We adopt the following notations. First remember that the \emph{wedge}  $\cP_{1} \wedge \cP_{2}$ of two partitions $\cP_{1},\cP_{2}$
is the coarsest  partition refining both $\cP_{1}$ and $\cP_{2}$, that is,   $\cP_{1} \wedge \cP_{2} = \{p_{1} \cap p_{2}, p_{1}  \in \cP_{1}, p_{2} \in \cP_{2}\}$.
If $\cP$ is a partition on some subset of $\cK \times \cQ$, such that every element $P$ of $\cP$ is the product of a clopen subset $a$ of $\cK$ by a clopen subset $b$ of $\cQ$, then we define a partition $\pi_{1*}(\cP)$ of $\cK$ as follows
 $$
 \pi_{1*}(\cP) := \mathop{\bigwedge}_{P=a \times b \in \cP} \left\{ a, \cK \setminus a \right\}.
 $$
This is the coarsest partition such that every set $\pi_{1}(P)$ with $P \in \cP$ can be expressed as the union of elements of the partition.

Now let us consider the following partition on $\cK$:
$$
\cV_{n} = \pi_{1*}(\cP(X_{n}))\wedge f^{-1}(\pi_{1*}(\cP(X_{n+1}))).
$$
Then we set
$$
\cX_{n} = \cP(X_{n}) \wedge \pi_{1}^{-1}(\cV_{n}) \mbox{ and } 
\cX'_{n+1} = \cP(X_{n+1}) \wedge \pi_{1}^{-1}(f(\cV_{n})).
$$ 
In other words, to get $\cX_{n}$ and $\cX'_{n+1}$, the elements of $\cP(X_{n})$ and $\cP(X_{n+1})$ have been cut off in the vertical direction, so that the projection on $\cK$ of the elements of  $\cX_{n}$ are unions of elements of a partition $\cV_n$ of $\cK$, and the projection  of the elements of  $\cX'_{n+1}$ are unions of elements of the image under $f$ of the same  partition $\cV_n$. This is the only property we will use, we do not need  the explicit form of the partition $\cV_{n}$.

Note that as usual the elements of our new partitions are products of clopen sets. For every element $R=a \times b \in \cX_{n}$, we consider the unique element $A \times B \in \cP_{\mid n \mid}$ that contains $R$, and  define the \emph{father} of $R$ to be the set 
$$
\widehat R := a \times B.
$$
Thus $\widehat R$ is obtained by enlarging $R$ vertically to the biggest possible rectangle refining $\cP_{\mid n \mid}$ (see figure~\ref{figure}).
Similarly we define the \emph{mother} of an element $R'=a \times b $ in $\cX'_{n+1}$ which is included in an element  $A \times B$ in the partition $\cP_{\mid n +1\mid}$ to be $\widehat{R'} := a \times B$.

\begin{figure}[htbp]
\includegraphics[width=15cm]{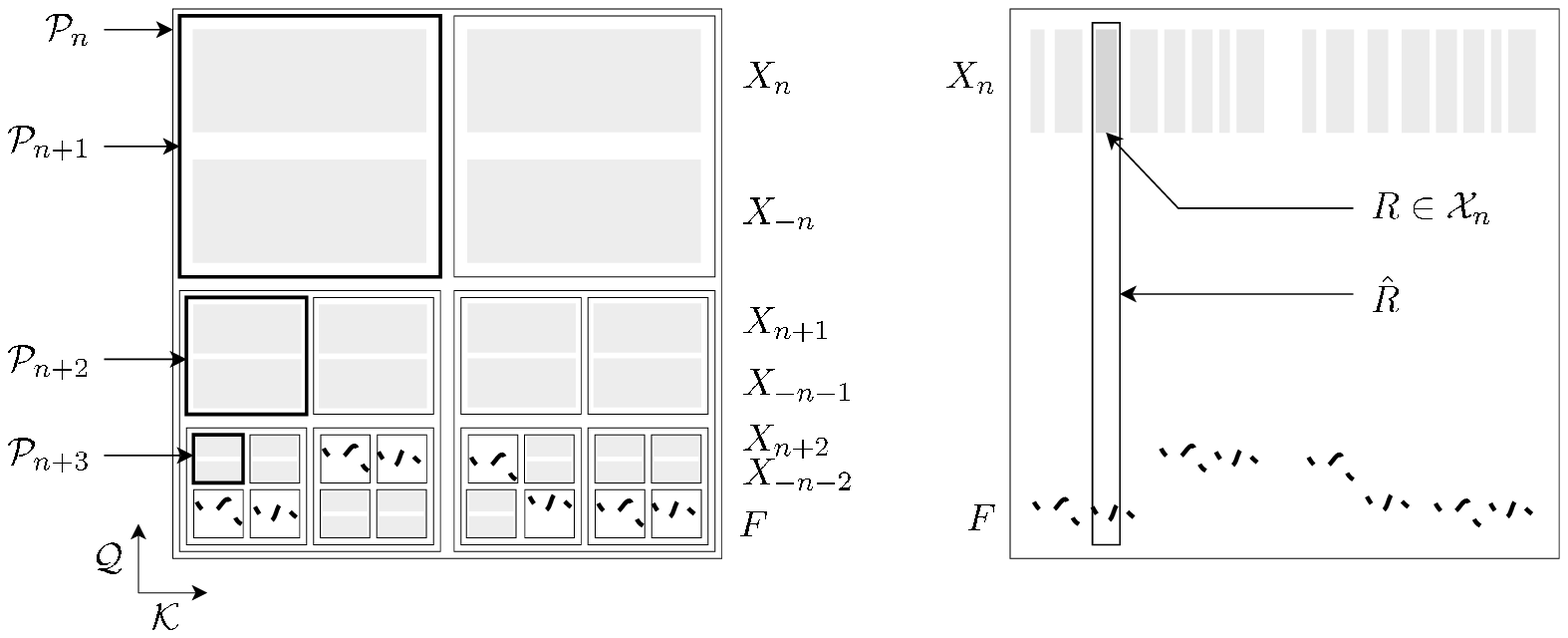}
\caption{The sequence $(X_{n})$ (left) and the definition of fathers (right)}
\label{figure}
\end{figure}

\bigskip

\noindent\textit{Claim~1. Every father and mother meet the set $F$.}

\bigskip

\begin{proof}[Proof of the claim]
Consider some element $R = a \times b$ of the partition $ \cX_{n}$, his father $\widehat R = a \times B$, and the element $P = A \times B$ of the partition $\cP_{\mid n\mid}$ as in the definition of $\widehat R$, with $a \subset A$. The set $R$ is included in $X_{n}$, thus also in $V_{\mid n \mid}$; the definition of $V_{\mid n \mid}$ shows that $P$ is also included in $V_{\mid n \mid}$ and meets $F$.
We now use property~\ref{pn.open} of the partition $\cP_{\mid n \mid}$, which says that $\pi_{1}(F \cap P) = \pi_{1}(P)$:
in other words, the set $\{\theta\}\times B$ meets $F$ for every $\theta \in \pi_{1}(P)$.
In particular $\widehat R = a \times B$ also meets $F$.
This proves the claim concerning fathers. The part concerning mothers is similar.
\end{proof}

We define a relation between the elements of $\cX_n$ and the elements of $\cX'_{n+1}$ as follows: an element $R$ is related to an element $R'$  if and only if $g_0(\widehat R)$ intersects $\widehat R'$.

\bigskip

\noindent \textit{Claim~2 (``No loneliness claim"). Every element $R$ of $\cX_n$ is in relation with at least one element $R'$ of $\cX'_{n+1}$. Every element $R'$ of $\cX'_{n+1}$ is in relation with at least one element $R$ of $\cX_{n}$.
Furthermore, if $R\in\cX_n$ is in relation with $R'\in\cX'_{n+1}$ then $f(\pi_{1}(R))=\pi_{1}(R')$.}

\bigskip

\begin{proof}
Consider an element $R$ of $\cX_{n}$ and his father $\widehat R$. According to claim~1, $\widehat R$ meets $F$. Since $F$ is invariant under $g_{0}$ we get that $g_{0}(\widehat R)$ still meets $F$. Thus to prove the first assertion of claim~2, we only need to see that the collection of all mothers of elements of $\cX'_{n+1}$ cover $F$. Similarly, the second assertion of claim~2 will follow from the fact that the collection of fathers of elements of $\cX_{n}$ cover $F$. Let us prove this last fact (the property concerning mothers follows from similar arguments).

Clearly the  partition $\cP_{\mid n\mid}$ covers $F$. Thus it suffices to show that any element $P=A \times B \in \cP_{\mid n \mid}$ meeting $F$ is a union of fathers of elements of $\cX_{n}$. Let us fix $\theta \in A$, we look for a father containing $\{\theta\} \times B$.
Here we make use of property~\ref{pn.last} of the partition $\cP_{\mid n\mid}$, which provides us with an element $\check P \in \cP_{\mid n\mid}+1$
such that 
\begin{enumerate}
\item $\theta \in \pi_{1}(\check P)$;
\item $\check P \cap F = \emptyset$.
\end{enumerate}
This last property implies that $\check P$ is included in $V_{\mid n \mid } \setminus V_{\mid n \mid +1}$, and thus $\check P$ has been divided horizontally into two sets, one of which is included in $X_{n}$ (the other one in $X_{-n}$). In particular there exists some  $q \in \cQ$ such that $X_{n} \cap \check P$ contains the point $(\theta,q)$. Let $R=a \times b$ be the element of the partition $\cX_{n}$ containing this point. We just have found a father $\widehat R = a \times B$ that contains $\{\theta\} \times B$, as desired.

We are left to prove the last assertion of claim~2. Suppose $R\in\cX_{n}$ is in relation with $R'\in\cX'_{n+1}$, that is, $g_{0}(\widehat R)$ meets $\widehat R'$.
Of course we also have $f(\pi_{1}(\widehat R))$ meets $\pi_{1}(\widehat R')$. By definition of the parents we have $f(\pi_{1}(\widehat R)) = f(\pi_{1}(R))$ and 
$\pi_{1}(\widehat R') = \pi_{1}(R')$. Now by definition of the partitions $\cX_{n}, \cX'_{n+1}$ both sets belongs to the partition $f(\cV_{n})$; since they meet they are equal. This completes the proof of the claim.
\end{proof}

\paragraph{The homeomorphism $g_1$} We first define the restriction of the homeomorphism $g_1$ to the set $X_n$, for every $n\in\ZZ$:

\bigskip

\noindent\textit{Claim~3. We can find a homeomorphism $g_1:X_n\to X_{n+1}$ such that:
\begin{itemize}
\item[--] $g_1$ fibres over $f$;
\item[--] $g_1$ is compatible with the above relation, in the following sense: for every $R\in \cX_n$ and every $R'\in \cX'_{n+1}$, if $g_1(R)$ meets $R'$ then $R$ is in relation with $R'$.
\end{itemize}}

\bigskip

\begin{proof}(``Divide and marry'')
We divide each set $R=a \times b \in \cX_{n}$ horizontally into as many pieces as the number of elements $X' \in \cX_{n+1}$ that are related to $R$,
thus getting  a partition of $R$ by the sets
$$
\{a \times b_{R'}, R \mbox{ is related to } R'\}.
$$
Similarly we consider for each $R'=a' \times b' \in \cX'_{n+1}$ a partition of $R'$,
$$
\{a' \times b'_{R}, R \mbox{ is related to } R'\}.
$$
Now for every related couple $(R=a \times b,R'=a'\times b')$ we define the map $g_{1}$ on $a\times b_{R'}$  by the formula
$$
g_{1}(\theta,q) = (f(\theta), G_{R,R'}(q))
$$
where $G_{R,R'}$ is any homeomorphism between the clopen sets $b_{R'}$ and $b'_{R}$.
Using  claim~2 we see that $g_{1}$ sends $a\times b_{R'}$ to $a'\times b'_{R}$, and that
this process defines a bijection from $X_{n}$ to $X_{n+1}$.
This map is clearly a homeomorphism which fibres over $f$ and is compatible with our relation.
\end{proof} 

We now define the homeomorphism $g_{1}$ on $\cK \times \cQ$ in the following way:
$g_{1}$ coincides with $g_{0}$ on $F$ and $g_{1}$ is given by claim~3 on each set $X_{n}, n \in \ZZ$. Since the $X_{n}$'s form a partition of the complement of $F$, we see that $g_{1}$ is a bijection. 
The continuity is obvious on the complement of $F$ since the sets $X_{n}$ are clopen. In order to prove the continuity of $g_1$ on $F$, we need the following claim.

\bigskip

\noindent\textit{Claim~4. Let $n \geq 0$. Then for every set $P \in \cP_{n}$ that meets $F$,
the set $g_{1}(P)$ is included in the union of the elements of $\cP_{n-1}$ that meets $g_{0}(P)$.
}

\bigskip

\begin{proof}
Let  $P \in \cP_{n}$ be a set  that meets $F$, we want to control the image of $P$. By construction, the map  $g_{1}$ coincides with $g_{0}$ on  $P \cap F$, thus on $P \cap F$ there is nothing to check. Let $x$ be a point in $P \setminus F$. Then there is some integer $m$ with $ | m  | \geq n$ such that $x$ belongs to  $X_{m}$.
Let $R$ be the element of the partition $\cX_{m}$ of $X_{m}$ that contains $x$.
By definition of $g_{1}$, the point $g_{1}(x)$ belongs to the mother $\widehat R'$ of some set $R' \in \cX_{m+1}$ which is related to $R$.
By definition of mothers, the set $\widehat R'$ is included in some element of the partition $\cP_{|m+1|}$, thus also in some element $P'$ of the partition $\cP_{n-1}$ since $|m+1| \geq n-1$. 
On the other hand the father $\widehat R$ of $R$ is included in $P$. By definition of the relation, the set $g_{0}(\widehat R)$ meets the set $\widehat R'$.
Hence the set $g_{0}(P)$ meets the set $P'$. The point $g_{1}(x)$ belongs to the element $P'$ of $\cP_{m-1}$ that meets $g_{0}(P)$, as desired.
\end{proof}

The continuity of $g_1$ on $F$ is now an immediate consequence of claim~4, together with property~2 of the sequence of partitions $(\cP_n)_{n\geq 0}$ and the continuity of $g_0$. Therefore $g_1$ is continuous everywhere. Then the continuity of $g_{1}^{-1}$ is automatic due to the compactness.
Finally let us note that  the sequence of sets $(X_n)_{n\in\ZZ}$ is a partition of $(\cK\times \cQ)\setminus F$ into clopen subsets, and by construction the  $g_{1}$ maps each $X_{n}$ onto $X_{n+1}$. Thus every point outside $F$ is a wandering point and the map $g_{1}$ satisfies the conclusion lemma~B. This completes the proof of the lemma.

\subsection*{Proof of theorem~\ref{theo.fibred-weiss} and~\ref{theo.bi-ergodic-fibred}}

Theorem~\ref{theo.bi-ergodic-fibred} is now an immediate consequence of propositions~\ref{p.bi-ergodic-ouvert} and~\ref{prop.open-fibred}:  the homeomorphism $g$ and the conjugating map $\Pi$ provided by proposition~\ref{p.bi-ergodic-ouvert} satisfy the hypotheses of proposition~\ref{prop.open-fibred}, and the diagram of theorem~\ref{theo.bi-ergodic-fibred} is obtained by concatenating the diagrams of propositions~\ref{p.bi-ergodic-ouvert} and~\ref{prop.open-fibred}. Similarly, theorem~\ref{theo.fibred-weiss} is an immediate consequence of theorem~\ref{theo.open} and proposition~\ref{prop.open-fibred}.


\part{Realisation of circle rotations on the two-torus}
\label{C}

In this part, we prove the existence of uniquely ergodic minimal realisations of circle rotations on the two-torus (theorem~\ref{theo.circle-rotation}). The construction has some additional properties that requires the following definition. If $\phi : \SS^1 \to \SS^1$ is a measurable function, then the set $\{(x, \phi(x)), x \in \SS^1\}$ is called a \emph{measurable graph}. If $\phi$ is continuous then this set is called a \emph{continuous graph}. Remember that the group $\mathrm{SL}(2,\RR)$ acts projectively on the circle. We will prove the following statement.
\begin{theodeuxbis}
\label{t.exotic-rotation}
For every $\alpha\in\RR\setminus\QQ$,  there exists a continuous map $A : \SS^1 \ra \mathrm{SL}(2,\RR)$ homotopic to a constant such that the skew-product homeomorphism $f:\TT^2\ra\TT^2$ defined by $f(x,y)=(x+\alpha, A(x).y)$ has the following properties:
\begin{itemize}
\item[--] $f$ is minimal; 
\item[--] $f$ is uniquely ergodic and the invariant measure $\mu$ is supported on a measurable graph.
\end{itemize}
\end{theodeuxbis}

If $f$ is a map given by theorem~\ref{t.exotic-rotation}  then the first coordinate projection $\pi_1:\TT^2\to\SS^1$ induces an isomorphism between $(\TT^2,\mu,f)$ and $(\SS^1,\mathrm{Leb}, R_{\alpha})$, where $R_{\alpha}$ is the rotation $x\mapsto x+\alpha$. Therefore, theorem~\ref{theo.circle-rotation} will follow from theorem~\ref{t.exotic-rotation}. The core of the proof of theorem~\ref{t.exotic-rotation} is the following technical lemma.
\begin{lemma*}
For every $\alpha\in\RR\setminus\QQ$,
and every $\varepsilon>0$, there exists a homeomorphism $g : \TT^2 \ra \TT^2$ with the following properties.
\begin{enumerate}
\item There exists a continuous map $m : \SS^1 \ra \mathrm{SL}(2,\RR)$ homotopic to a constant, such that
$$
g = M^{-1} \circ ( R_{\alpha} \times \mathrm{Id} ) \circ M
$$
where $M(x,y)=(x,m(x).y)$. In particular, the homeomorphism $g$ is a skew-product over the circle rotation $R_\alpha$, and is conjugated to $R_\alpha\times\mbox{Id}$.
\item The homeomorphism $g$ is $\varepsilon$-close to $R_{\alpha} \times \mathrm{Id}$ in the $C^0$-topology.
\item Every $g$-invariant continuous graph $C$ is $\varepsilon$-dense in $\TT^2$.
\item There exists a horizontal open strip
 $\Gamma=\{(x,y)\in\TT^2 \mid y\in V_x\}$ of width $\varepsilon$ (by such we mean that $V_x$ is an interval of length $\varepsilon$ depending continuously on $x$) such that, for every $g$-invariant continuous graph $C=\{(x,\phi(x))\mid x\in\SS^1\}$, one has
$$\mbox{Leb}\left(\pi_1(C\cap\Gamma)\right)=\mbox{Leb}\left(\{x\in\SS^1 \mid \phi(x)\in V_x\}\right)\geq 1-\epsilon.$$
\end{enumerate}
\end{lemma*}

\subsection*{Proof of the lemma}

The following fact, which is a refinement of the classical Rokhlin lemma, will be applied to the circle rotation $R_{\alpha}$.
\begin{fact}
Let us consider a compact manifold $X$, a homeomorphism $h$ of $X$ and a $h$-invariant probability measure $\mu$ which has no atom.  Let $N,k$ be any positive integers and $\epsilon$ be any positive real number. Then there exist some compact sets $D_1,\dots,D_k\subset X$ such that: 
\begin{enumerate}
\item the sets $h^i(D_j)$ for $0\leq i \leq N$ and $1\leq j \leq k$ are pairwise disjoint;
\item the union of the $h^i(D_j)$ for $0\leq i \leq N$ and $1\leq j \leq k$ has $\mu$-measure larger than $1-\varepsilon$;
\item for any $j_0\in \{1,\dots,k\}$, the $\mu$-measure of the union $\bigcup_{i=0}^N h^i(D_{j_0})$ is smaller than $k^{-1}$.
\end{enumerate}
\end{fact}

\begin{proof}
Rokhlin lemma ensures the existence of a measurable set $A\subset X$ such that the iterates $h^i(A)$
for $0\leq i \leq N$ are pairwise disjoint and their union has $\mu$-measure larger than $1-\varepsilon$. Then, using the fact that the measure $\mu$ is regular, we can find a compact set $B\subset A$ such that these properties are is still satisfied when we replace $A$ by $B$.
One can assume that $B$ is the union of disjoint compact sets with arbitrarily small measure:
this allows to decompose $B$ as a disjoint union $D_1\cup \dots\cup D_k$ of compact sets whose $\mu$-measure
is close to $\mu(B)/k$. These sets satisfy the conclusion of the fact.
\end{proof}

Let us come to the proof of the lemma. We see the torus $\TT^2$ as the product of two copies of $\SS^1$. For sake of clarity, we shall distinguish between these two copies, denoting them respectively by $\SS^1_h$ and $\SS^1_v$ (where ``$h$" and ``$v$" stand for ```horizontal" and ``vertical"). The rotation $R_{\alpha}$ acts on $\SS^1_{h}$, whereas the elements of $\mbox{SL}(2,\RR)$ defined below act on $\SS^1_{v}$. 

\paragraph{Construction of the homeomorphism $g$.} According to the first item of the lemma
we have to construct a continuous map \hbox{$m:\SS^1_{h}\to\mbox{SL}(2,\RR)$} homotopic to a constant.
We proceed step by step.

\begin{itemize}
\item We first choose a finite collection of intervals $A_1,\dots,A_k\subset\SS^1_{v}$ such that
\begin{itemize}
\item for every $j$, the length of the interval $A_j$ is less than $\epsilon/2$;
\item the union $A_1\cup\dots\cup A_k$ covers the whole circle $\SS^1_{v}$.
\end{itemize}
Note that in particular one has $k\varepsilon/2\geq 1$.

\item We choose some pairwise disjoint intervals $R_1,\dots,R_k\subset\SS^1_{v}$ such that, for each $j\in\{1,\dots,k\}$, the interval $R_j$ is disjoint from the interval $A_j$. 

\item For each $j\in\{1,\dots,k\}$, we choose a hyperbolic map 
$S_j\in\mathrm{SL}(2,\RR)$ whose attractive fixed point is included in the interior of $A_{j}$ and whose repulsive fixed point is included in the interior of  $R_{j}$,
and which is $\epsilon$-close to the identity of $\SS^1_{v}$ (for the $C^0$ topology). Note that  $S_j$  maps $A_j$ into itself.
Then  we choose an integer $\ell_j$ such that  $S_j^{\ell_j}$ maps $\SS^1_{v}\setminus R_j$ into $A_j$.
We set $\displaystyle \ell:=\max_{j\in\{1\dots k\}} \ell_j$. So, for every $j$, the homeomorphism $S_j^\ell$ maps $\SS^1_{v}\setminus R_j$ into $A_j$.

\item We choose an integer $N>8\ell/\varepsilon$ large enough so that for every $x\in\SS^1_{h}$, the orbit segment $x,R_\alpha(x),\dots,R_\alpha^{N-2\ell}(x)$ is $(\epsilon/2)$-dense in~$\SS^1_{h}$.

\item We choose by the above fact a finite number of compact sets
$D_1,\dots,D_k\subset\SS^1_{h}$ such that:
\begin{itemize}
\item the sets $R_\alpha^i(D_j)$ for $j=1,\dots,k$ and $i=0,\dots,N$ are pairwise disjoint;
\item the Lebesgue measure of the union $\bigcup_{j=1}^k\bigcup_{i=0}^N R_\alpha^i(D_j)$ is larger than $1-\epsilon/4$;
\item for each $j_0\in \{1,\dots,k\}$, the measure of $\bigcup_{i=0}^N R^i_\alpha(D_{j_0})$
is smaller than $k^{-1}\leq \varepsilon/2$.
\end{itemize}

\item For $j\in \{1,\dots,k\}$ and $i\in \{0,\dots,N\}$
we define the map $m$ on $R^i_\alpha(D_j)$ by
$$
m:=\left\{
\begin{array}{ll}
 S_j^i & \mbox{if }0\leq i\leq \ell\\
 S_j^\ell & \mbox{if }\ell\leq i\leq N-\ell\\
 S^{N-i}_j & \mbox{if }N-\ell\leq i\leq N.
\end{array}
\right.$$

Note that $m(x+\alpha)^{-1}m(x)$ is $\varepsilon$-close to the identity for every $x\in \bigcup_{j=1}^{k}\bigcup_{i=0}^{N-1} R_\alpha^i(D_j)$. 
Now, we extend continuously the map $m$ on $\SS^1_h$ with the constraints that
$m$ is homotopic to a constant and $m(x+\alpha)^{-1}m(x)$ is $\varepsilon$-close
to the identity for every $x\in\SS^1_h$.
This can be done as follow. For each $j\in \{1,\dots,k\}$,
we choose a small neighbourhood $U_j$ of $D_j$ and
a continuous map $\phi\colon U_j \to SL(2,\RR)$, $\varepsilon$-close to the identity of $\SS^1_v$, equal to $S_{j}$ on $D_{j}$,
and that coincides with the identity on the boundary of $U_j$.
Then we set for $x\in R^i_\alpha(U_j)$,
$$
m(x):=\left\{
\begin{array}{ll}
 (\varphi(R_\alpha^{-i}(x)))^i & \mbox{if }0\leq i \leq \ell\\
 (\varphi(R_\alpha^{-i}(x)))^\ell & \mbox{if }\ell\leq i\leq N-\ell\\
 (\varphi(R_\alpha^{-i}(x)))^{N-i} & \mbox{if }N-\ell\leq i\leq N.
\end{array}
\right.$$
If the neighbourhoods $U_j$ are chosen small enough then the sets
$R_\alpha^i(U_j)$ for  $0 \leq i \leq N, 1\leq j\leq k$ are pairwise disjoint and the above formulae do make sense. For the points $x$ that do not belong to one of these sets, $m(x)$ is defined to be  the identity map of $\SS^1_v$.

\item We choose an open interval $V_x\subset \SS^1_v$ of length $\varepsilon$ which depends continuously
on $x\in \SS^1_h$ and such that for every $1\leq j\leq k$
it contains $A_j$ whenever $x$ belongs to $R^i_\alpha(D_j)$
for some $0\leq i\leq N$.
Then we consider the horizontal strip $\Gamma:=\{(x,y) \mid y\in V_x\}$.
\end{itemize}

\paragraph{Properties of the homeomorphism $g$.}
Let us check that the maps $g$ and $M$ associated to $m$
as in the statement of the lemma satisfy the required properties:

\begin{enumerate}

\item The first item of the lemma is a consequence of the definition of $g$.

\item Since the map $x\mapsto m(x+\alpha)^{-1}m(x)$ is $\varepsilon$-close to the identity map of the circle $\SS^1_{v}$, the homeomorphism $g:(x,y)\mapsto (x+\alpha \; , \; m(x+\alpha)^{-1}m(x).y)$ is $\varepsilon$-close to $R_\alpha\times\mbox{Id}$. 

\item Let $C$ be a  $g$-invariant  continuous graph. Then $M(C)$ is a $(R_\alpha\times\mbox{Id})$-invariant continuous graph. Hence there exists a point $y\in\SS^1_{v}$ such that $C=M(\SS^1_{h}\times \{y\})$.
Let $j\in\{1,\dots,k\}$ be an integer such that $y\notin R_j$.
We claim that $C$ is $\varepsilon/2$-dense in $\SS^1_{h}\times A_j$.
Indeed consider a point $x$ in $D_j$.
On the one hand, for every $i\in\ZZ$, the point $M(R_\alpha^i(x),y)$ belongs to the graph $C$.
On the other hand, for $i\in \{\ell,\dots,N-\ell\}$, by our choice of $x$ we have 
$$
M\left(R_\alpha^i(x),y\right) = \left(R_\alpha^i(x), m\left(R_\alpha^i(x)\right).y\right) = \left(R_\alpha^i(x), S_j^\ell.y\right) 
$$
and this point belongs to $\{R_\alpha^i(x)\}\times A_j$ by our choice of $S_{j}$ and since $y \not \in R_{j}$.
Now remember that the length of the interval $A_j$ is less than $\epsilon/2$ and that the integer $N$ was chosen in such a way that the sequence $R_\alpha^{\ell}(x),\dots,R_\alpha^{N-\ell}(x)$ is $\epsilon/2$-dense in $\SS^1$. This shows the claim.

Since the intervals $R_j$ are pairwise disjoint, there is at most one integer $j_0\in\{1,\dots,k\}$ such that $y \in R_{j_0}$. By construction the union $\bigcup_{j\neq j_0} A_j$ is $\epsilon/2$-dense in $\SS^1_{v}$. Therefore the graph $C$ is $\epsilon$-dense in $\TT^2=\SS^1_{h}\times\SS^1_{v}$.

\item Let us once again consider a $g$-invariant continuous graph $C=M(\SS^1_h\times \{y\})$.
As above for every $i\in \{\ell,\dots,N-\ell\}$ and $j\in \{1,\dots, k\}$
such that $y\notin R_j$, the point $M\left(R_\alpha^i(x),y\right)$ belongs to  the set $\{R_\alpha^i(x)\}\times A_j$ which  is included in the strip $\Gamma$.
There is at most one $j_0\in \{1,\dots,k\}$ such that $y\in R_{j_0}$. Hence,
$$\pi_1(C\cap\Gamma)\supset\bigcup_{j\notin j_0}\bigcup_{i=\ell}^{N-\ell} R_\alpha^i(D_j).$$
As a consequence, we get
\begin{eqnarray*}
\mbox{Leb}\left(\pi_1(C\cap\Gamma)\right) & \geq & \frac{N-2\ell}{N}
\left(\mbox{Leb}\left(\bigcup_{j=1}^k\bigcup_{i=1}^{N} R_\alpha^i(D_j)\right)
- \mbox{Leb}\left(\bigcup_{i=1}^{N} R_\alpha^i(D_{j_0})\right)\right) \\
& \geq & (1-\epsilon/4)(1-\epsilon/4-\varepsilon/2)\geq 1-\varepsilon
\end{eqnarray*}
(the second inequality follows from the definition of the intervals $D_1,\dots,D_k$ and the integer $N$).
\end{enumerate}

This completes the proof of the lemma.

\subsection*{Proof of theorem~\ref{t.exotic-rotation}}

We will use the lemma to construct inductively a sequence of homeomorphisms $(M_{k})_{k \geq 0}$.
This will give rise to the sequences $(\Phi_{k})_{k \geq 0}$ and $(f_{k})_{k \geq 0}$
defined by
$$
\Phi_{k}= M_{0} \circ \cdots \circ M_{k} \mbox{ and } f_{k} = \Phi_{k} \circ (R_{\alpha} \times \mathrm{Id})
\circ \Phi_{k}^{-1}.
$$

We set $M_{0} = \mathrm{Id}$. Assuming that the sequence $M_{0}, \dots , M_{k}$ has been constructed, we consider a small positive number $\varepsilon_{k+1}$ (the conditions on $\varepsilon_{k+1}$ will be detailed below) and we apply the lemma to $\varepsilon=\varepsilon_{k+1}$.
The lemma provides the maps $g_{k+1}$ and $M_{k+1}$, and 
 we have
$$f_{k+1} = \Phi_{k} \circ M_{k+1} \circ (R_{\alpha} \times \mathrm{Id}) \circ M_{k+1}^{-1} \circ \Phi_{k}^{-1} =    \Phi_{k} \circ g_{k+1} \circ  \Phi_{k}^{-1}.$$
By the second item of the lemma the homeomorphisms $g_{k+1}$ and $R_{\alpha} \times \mathrm{Id}$
are $\varepsilon_{k+1}$-close, so if $\varepsilon_{k+1}$ has been chosen small enough,
then $f_{k+1}$ is $2^{-k}$-close to 
$f_{k} = \Phi_{k} \circ  (R_{\alpha} \times \mathrm{Id}) \circ  \Phi_{k}^{-1}$.
One can thus assume that the sequence $(f_{k})$ is a Cauchy sequence.
It converges to a homeomorphism $f$ of the two-torus,
which will be a skew-product of the form required by the theorem (because of the first item of the lemma
and since $SL(2,\RR)$ is closed in the space of circle homeomorphisms).
It remains to check that, as soon as the sequence $(\varepsilon_{k})$ decreases sufficiently fast,
the map $f$ satisfies the conclusions of the theorem.

Let us address the minimality. By item 3 of the lemma, at step $k$ every $g_{k}$-invariant continuous graph is $\varepsilon_{k}$-dense in $\TT^2$. Consequently we can choose $\varepsilon_{k}$ small enough so that every $f_{k}$-invariant graph is $1/k$-dense in $\TT^2$. Since $f_{k}$ is conjugate to $R_{\alpha} \times \mathrm{Id}$,
every orbit is dense in an invariant circle, and there exists a positive integer $N_{k}$ such that  every piece of orbit $p, f_{k}(p), \dots , f^{N_{k}}(p)$ is again $1/k$-dense in $\TT^2$.
This last property is open: the  sequence $(\varepsilon_{\ell})_{\ell > k}$ can be chosen so that  this property is shared by the limit map  $f$, namely,  
every piece of $f$-orbit of length $N_{k}$ is $1/k$-dense. This entails the minimality of $f$.

We now turn to the ergodic properties.
Let $\Gamma$ be the horizontal strip given by item 4 of the lemma when applied at step $k$
and let $\Gamma_{k } := \Phi_{k-1}(\Gamma)$.
Then $\Gamma_{k}$ is a horizontal strip: the  intersection of $\Gamma_{k}$ with any circle
$\{x\}\times \SS^1_{v}$ is an interval, whose length is less than $1/k$ if $\varepsilon_{k}$ is small enough.
Furthermore, for every $f_{k}$-invariant continuous graph $C$, the  set 
 $\pi_{1}(C \cap \Gamma_{k})$ is an open subset of the circle whose Lebesgue measure is bigger than  $1-\varepsilon_{k}$. Then the unique ergodicity of the rotation $R_{\alpha}$ entails the following fact.
\begin{fact}
There exists a positive integer $N_{k}$ with the following property:
 every $f_{k}$-orbit of length  $N_{k}$ spends more than a ratio $(1-\varepsilon_{k})$ of its time within the strip $\Gamma_{k}$, in other words
for every  point $p$, 
$$
\mathrm{Card} (\{ 0 \leq n < N_{k} , f_{k}^n(p) \in \Gamma_{k} \})
 >  (1-\varepsilon_{k})N_{k}.
$$

\end{fact}
Since $\Gamma_k$ is open, 
every point $p$ has a neighbourhood $V_{p}$ such that this inequality remains true when we replace $p$ by any  point $p' \in V_{p}$ and $f_{k}$ by any map which is $\varepsilon_{p}$-close to $f_{k}$ (for some positive $\varepsilon_{p}$). Thus, by compactness, the property expressed in the fact is open, and shared by the map $f$ as soon as the sequence
$(\varepsilon_{\ell})_{\ell > k}$ tends to $0$ fast enough.
As a consequence, for any $f$-invariant measure $\mu$ one has
$$\mu(\Gamma_{k}) \geq 1-\varepsilon_{k}.$$

For any positive integer $k_{0}$ we set
$$
\cC_{k_{0}} = \bigcap_{i \geq k_{0}}  \Gamma_{i}
\ \ \ \ \ \mbox{ and } \ 
\cC= \bigcup_{k_{0}\geq 0} \cC_{k_{0}}.
$$
For any $f$-invariant measure $\mu$,
the measure of $\cC_{k_0}$ is bounded from below by
$1-\sum_{k\geq k_0}\varepsilon_k$
and is positive if the sequence $(\varepsilon_{k})$ goes to $0$ fast enough;
hence $\cC$ has measure $1$.
Remember also that $\Gamma_{k}$ is a strip of thickness less than $1/k$.
Thus the intersection of $\cC$ with any vertical circle is empty or reduced to a point:
$\cC$ is a measurable graph over a set of full Lebesgue measure.
The unique ergodicity of the rotation $R_{\alpha}$ implies that $f$ is also uniquely ergodic, the only invariant measure being the measure $\mu$ defined by the formula
$$
\mu(E) = \mathrm{Leb}(\pi_{1}(E \cap \cC)).
$$
This completes the proof of theorem~\ref{t.exotic-rotation}.




\begin{thebibliography}{150}

\bibitem{AnoKat}
Anosov, Dmitri and Katok, Anatole.
New examples in smooth ergodic theory. Ergodic diffeomorphisms.
\textit{Transactions of the Moscow Mathematical Society}
\textbf{23} (1970), 1--35.


\bibitem{BegCroLeR}
B\'eguin, Fran\c cois; Crovisier, Sylvain and Le~Roux, Fr\'ed\'eric.
Construction of curious minimal uniquely ergodic homeomorphisms on manifolds: the Denjoy-Rees technique.
\textit{Ann. Sci. \'Ecole Norm. Sup.} \textbf{40} (2007), 251--308.


\bibitem{FayKat}
Fayad, Bassam and Katok, Anatole. 
Constructions in elliptic dynamics.  \emph{Ergod. Th.  Dyn. Sys.} \textbf{24}  (2004), 1477--1520.

\bibitem{Glasner}
Glasner, Eli. \textit{Ergodic theory via joinings.}
Mathematical Surveys and Monographs \textbf{101}. AMS, Providence (2003).

\bibitem{GlaWei}
 Glasner, Eli and Weiss, Benjamin. On the interplay between measurable and topological dynamics. \emph{Handbook of dynamical systems},
  Vol. 1B, 597--648, Elsevier, Amsterdam (2006).


\bibitem{Hansel-Raoult}
Hansel, Georges and Raoult, Jean-Pierre.
Ergodicity, uniformity and unique ergodicity.
\textit{Indiana Univ. Math. J.} \textbf{23} (1973/74), 221--237.


\bibitem{Jewett}
Jewett, Robert I. The prevalence of uniquely ergodic systems. \emph{J. Math. Mech.} \textbf{19} (1969/1970), 717--729. 


\bibitem{Kechris}
Kechris, Alexander S. \textit{Classical descriptive set theory.}
Graduate Texts in Mathematics \textbf{156}. Springer-Verlag, New York (1995).

\bibitem{Krieger}
Krieger, Wolfgang. On unique ergodicity. \emph{Proceedings of the Sixth Berkeley Symposium on Mathematical Statistics and Probability 1970/1971},
Vol. II: Probability theory,  327--346, Univ. California Press, Berkeley, Calif. (1972). 



\bibitem{Skl}
\v Sklover, M. D.
Classical dynamical systems on the torus with continuous spectrum. 
\textit{Izv. Vys\v s. U\v cebn. Zaved. Matematika} \textbf{10} (1967), 113--124. 

\bibitem{Weiss}
Weiss, Benjamin. Strictly ergodic models for dynamical systems.
\emph{Bull. Amer. Math. Soc. (N.S.)} \textbf{13} (1985), 143--146.



\end{thebibliography}
\end{document}